\documentclass{amsart}
\date{05 March 2013}
\usepackage{latexsym,amsthm,amsmath,amsfonts,amscd,amssymb}
\usepackage{hyperref}
\usepackage[ansinew]{inputenc}
\usepackage[all]{xy}
\usepackage{graphics}
\usepackage{lscape}
\usepackage{array}
\theoremstyle{plain}  
\newtheorem{theorem}{Theorem}[section]

\newtheorem*{theorem*}{Theorem}

\newtheorem{corollary}[theorem]{Corollary}
\newtheorem{lemma}[theorem]{Lemma}
\newtheorem{proposition}[theorem]{Proposition}

\newtheorem{tech-lemma}[theorem]{Technical Lemma}
\newtheorem{definition}[theorem]{Definition}
\theoremstyle{remark}

\newtheorem{remark}[theorem]{Remark}
\newtheorem*{remark*}{Remark}

\newtheorem*{claim*}{Claim}
\numberwithin{equation}{section}
\renewcommand{\leq}{\leqslant}

\renewcommand{\geq}{\geqslant}

\newcommand{\RR}{\mathbb{R}}
\newcommand{\Z}{\mathbb{Z}}
\newcommand{\CC}{\mathbb{C}}
\newcommand{\HH}{\mathbb{H}}
\newcommand{\M}{{\mathcal M}}
\newcommand{\N}{{\mathcal N}}

\newcommand{\U}{\mathrm{U}}
\newcommand{\GL}{\mathrm{GL}}

\newcommand{\Sp}{\mathrm{Sp}}

\newcommand{\st}{\;|\;}

\DeclareMathOperator{\ad}{ad}
\DeclareMathOperator{\Ad}{Ad}
\DeclareMathOperator{\Aut}{Aut}

\DeclareMathOperator{\rk}{rk}

\DeclareMathOperator{\Hom}{Hom}
\DeclareMathOperator{\End}{End}

\newcommand{\liem}{\mathfrak{m}}

\newcommand{\liemc}{\mathfrak{m}^{\mathbb{C}}}
\newcommand{\lieh}{\mathfrak{h}}
\newcommand{\liehc}{\mathfrak{h}^{\mathbb{C}}}
\newcommand{\lieg}{\mathfrak{g}}
\newcommand{\liegc}{\mathfrak{g}^{\mathbb{C}}}

\let\oldmarginpar\marginpar
\renewcommand\marginpar[1]{\oldmarginpar{\tiny\bf\begin{flushleft} #1
\end{flushleft}}}

\begin{document}

%
%

\title[Connectednes of the moduli of $\Sp(2p,2q)$-Higgs bundles]
{Connectedness of the moduli of $\Sp(2p,2q)$-Higgs bundles}
%
%

\author[O. García-Prada]{Oscar García-Prada}
\address{O. García-Prada\\ Instituto de Ciencias Matemáticas\\
CSIC-UAM-UC3M-UCM\\
Calle Nicol\'as Cabrera 15\\
28049 Madrid\\
Spain.}
\email{oscar.garcia-prada@icmat.es}

\author[A. G. Oliveira]{André G. Oliveira}
\address{A. G. Oliveira\\ Departamento de Matemática\\
  Universidade de Trás-os-Montes e Alto Douro\\
  Quinta dos Prados\\ 5000-911 Vila Real \\ Portugal.}
\email{agoliv@utad.pt}

\thanks{%
First author partially supported by the Ministerio de Econom\'{\i}a y
Competitividad of Spain through Project MTM2010-17717 and Severo Ochoa
Excellence Grant.
  Second author partially supported the FCT (Portugal) with national funds 
through the projects PTDC/MAT/099275/2008, PTDC/MAT/098770/2008 and PTDC/MAT/120411/2010 and through Centro de Matem\'atica da
Universidade de Tr\'as-os-Montes e Alto Douro (PEst-OE/MAT/UI4080/2011).
Both authors thank the Centre de Recerca Matematica in Barcelona - that they visited while preparing the paper - for the excellent conditions provided.}

\begin{abstract}
Using the Morse-theoretic techniques introduced by Hitchin, we prove that 
the moduli space of $\Sp(2p,2q)$-Higgs bundles over a compact Riemann surface 
of genus $g\geq 2$ is connected. In particular, this implies that the
moduli space of representations of the fundamental group of the surface 
in $\Sp(2p,2q)$ is connected. 
\end{abstract}

\maketitle

%
%

\section{Introduction}\label{section:Introduction}

Let $X$ be a compact Riemann surface, and let
$\M_G$ be the moduli space of polystable $G$-Higgs bundles over $X$, 
where $G$ is a real reductive Lie group. Higgs bundles were 
first introduced by Nigel Hitchin in \cite{hitchin:1987} to be a pair $(V,\varphi)$ consisting of a holomorphic bundle $V$ over $X$ 
and a holomorphic section $\varphi$ of the bundle $\End V$ twisted with the
canonical bundle of $X$. This notion was then generalised to that of a $G$-Higgs bundle \cite{hitchin:1992,bradlow-garcia-prada-gothen:2005}, so that Hitchin's original definition is a $\GL(n,\CC)$-Higgs bundle.
In this paper we study the moduli space $\M_{\Sp(2p,2q)}$,
where $\Sp(2p,2q)$ is the real form of $\Sp(2p+2q,\CC)$ defined by the involution $M\mapsto K_{p,q}{M^*}^{-1}K_{p,q}$ on $\Sp(2p+2q,\CC)$, where 
$$K_{p,q}=\begin{pmatrix}
     -I_p & 0 & 0 & 0 \\
     0 & I_q & 0 & 0 \\
     0 & 0 & -I_p & 0 \\
     0 & 0 & 0 & I_q
\end{pmatrix},$$ $I_p$ and $I_q$ being the identity matrices of the given type.

In this paper we prove  the following.
\begin{theorem*}
Suppose that $X$ has genus at least $2$. The moduli space $\M_{\Sp(2p,2q)}$ of $\Sp(2p,2q)$-Higgs bundles over $X$ is connected.
\end{theorem*}

We adopt the Morse-theoretic techniques introduced by Hitchin in
\cite{hitchin:1987}, which reduce the question to the study of connectedness of certain subvarieties of $\M_{\Sp(2p,2q)}$, defined as the subvariety of local minima of the so-called Hitchin function, defined on $\M_{\Sp(2p,2q)}$. These techniques proved to be extremely efficient in the calculation of the connected components of the moduli spaces of $G$-Higgs bundles for several other groups (see, for example, \cite{hitchin:1992,bradlow-garcia-prada-gothen:2003,bradlow-garcia-prada-gothen:2005,garcia-gothen-mundet:2008 II}). In order to apply this method, we first obtain a detailed description of the smooth points of the moduli space $\M_{\Sp(2p,2q)}$ corresponding to the stable and simple $\Sp(2p,2q)$-Higgs bundles, and prove that the only local minima of the Hitchin function of this type are the ones with zero Higgs field. Then 
we show that stable and non-simple $\Sp(2p,2q)$-Higgs bundles are always given by no-trivial direct sums of stable and simple $\Sp(2p,2q)$-Higgs bundles. This implies again that the only stable and non-simple local minima of the Hitchin function must have zero Higgs field. Finally we deal with strictly polystable $\Sp(2p,2q)$-Higgs showing that they split as a direct sum of stable Higgs bundles for one of the following groups: $\Sp(2p_\alpha,2q_\alpha)$, $\U(p_\alpha,q_\alpha)$, $\Sp(2n_\alpha)$ or $\U(n_\alpha)$, where $p_\alpha\leq p$, $q_\alpha\leq q$ and $n_\alpha\leq p+q$.
Hence the question of finding all strictly polystable local minima
$\Sp(2p,2q)$-Higgs bundles amounts to the same question but for the other
given groups.  And here we are confronted to a situation, which as far as we
know, is the first time that appears in the study of the connectedness 
properties of moduli spaces of Higgs bundles using these methods.
Namely,  for the subgroup  $\U(p_\alpha,q_\alpha)\subset\Sp(2p,2q)$ the
Higgs bundles which are local minima of the corresponding Hitchin function 
have non-zero Higgs field in general. To deal with this situation, and show
that these strictly polystable objects are not local minima for $\Sp(2p,2q)$, 
we give a direct argument, providing a deformation to a stable object, 
and using the fact that stable $\Sp(2p,2q)$-Higgs bundles with non-zero Higgs field are not local minima.

For a semisimple Lie group $G$, non-abelian Hodge theory on $X$ establishes a homeomorphism between $\M_G$ and the moduli space of reductive representations of $\pi_1X$ in $G$ (cf. \cite{hitchin:1987,simpson:1988,simpson:1992,donaldson:1987,corlette:1988,garcia-gothen-mundet:2008,bradlow-garcia-prada-mundet:2003}). A direct consequence of our result is thus the following.
\begin{theorem*}
The moduli space of reductive representations of $\pi_1X$ in $\Sp(2p,2q)$ is connected.
\end{theorem*}

We finally mention that the main results of this paper are consistent with the recent results independently obtained by Laura Schaposnik in her DPhil Thesis \cite{schaposnik:2012}, using other methods, namely through the study of the Hitchin map.

\section{$\Sp(2p,2q)$-Higgs bundles}\label{section:Sp(2p,2q)-Higgs bundles}

Let $X$ be a compact Riemann surface of genus $g\geq 2$, and let $G$ be a real reductive Lie group, which is a real form of some complex reductive Lie group $G^\CC$.
Let $H\subseteq G$ be a maximal compact subgroup
so that its complexification $H^\CC$ is a closed subgroup of $G^\CC$. Denote by $\lieg$ and $\lieh$ the Lie algebras of $G$ and $H$, and let
$\lieg=\lieh\oplus\liem$ be a Cartan
decomposition of $\lieg$,
where $\liem$ is the complement of $\lieh$ with
respect to a non-degenerate $\Ad(G)$-invariant bilinear form on $\lieg$. If
$\theta:\lieg\to\lieg$ is the
corresponding Cartan involution then $\lieh$ and
$\liem$ are its $+1$-eigenspace and $-1$-eigenspace,
respectively. Complexifying, we have the decomposition
$\liegc=\liehc\oplus\liemc$ and
$\liemc$ is a representation of $H^\CC$ through the
so-called \emph{isotropy representation}
$\iota:H^\CC\to\Aut(\liemc)$, induced by the adjoint
representation of $G^\CC$ on $\liegc$.
If $E_{H^\CC}$ is a principal $H^{\CC}$-bundle over $X$, we denote by
$E(\liemc)=E_{H^\CC}\times_{H^{\CC}}\liemc$ the
vector bundle, with fibre $\liemc$, associated to the
isotropy representation.

Let $K:=T^*X^{1,0}$ be the canonical line bundle of $X$.
\begin{definition}\label{definition of Higgs bundle}
A \emph{$G$-Higgs bundle} over the compact Riemann surface $X$ is a pair
$(E_{H^\CC},\varphi)$ where $E_{H^\CC}$ is a principal holomorphic $H^\CC$-bundle over
$X$ and $\varphi$ is a global holomorphic section of
$E(\liemc)\otimes K$, called the \emph{Higgs field}.
\end{definition}

Let us focus on $G=\Sp(2p,2q)$. In intrinsic terms, this is the group of quaternionic linear automorphisms of an $p+q$-dimensional vector space $V$ over the ring $\mathbb{H}$ of quaternions, which preserve a hermitian form on $V$ with signature $(2p,2q)$.

In terms of matrices, $\Sp(2p,2q)$ is the subgroup of the complex symplectic group $\Sp(2p+2q,\CC)$ defined as
$$\Sp(2p,2q):=\left\{M\in\Sp(2p+2q,\CC)\st M^*K_{p,q} M=K_{p,q}\right\},$$ where $M^*$ denotes the conjugate transpose of $M$,
$$K_{p,q}:=\begin{pmatrix}
     -I_p & 0 & 0 & 0 \\
     0 & I_q & 0 & 0 \\
     0 & 0 & -I_p & 0 \\
     0 & 0 & 0 & I_q
\end{pmatrix}$$ and $I_p$ and $I_q$ are the identity matrices of the corresponding size.
From this definition, it is obvious that $\Sp(2p,2q)$ is a real form of $\Sp(2p+2q,\CC)$ given by the fixed point set of the involution $M\mapsto K_{p,q}{M^*}^{-1}K_{p,q}$ on $\Sp(2p+2q,\CC)$. Let $\mathfrak{sp}(2p,2q)$ denote the Lie algebra of $\Sp(2p,2q)$. If $\sigma$ is the involution of the Lie algebra $\mathfrak{sp}(2p+2q,\CC)$ defining the real form $\mathfrak{sp}(2p,2q)$, then
$$\sigma(M)=-K_{p,q}M^*K_{p,q},$$
and if $\tau:\mathfrak{sp}(2p+2q,\CC)\to\mathfrak{sp}(2p+2q,\CC)$ is the involution defining the compact form, $\mathfrak{sp}(p+q)$, then 
$$\tau(M)=-M^*.$$
Since $\tau$ and $\sigma$ commute, define the Cartan involution 
$\theta:\mathfrak{sp}(2p+2q,\CC)\to\mathfrak{sp}(2p+2q,\CC)$ by
\begin{equation}\label{eq:Cartan involution}
\theta(M):=\sigma\tau(M)=K_{p,q}MK_{p,q}.
\end{equation}
The corresponding Cartan decomposition of the complex Lie algebra is 
$$\mathfrak{sp}(2p+2q,\CC)=\mathfrak{sp}(2p,\CC)\oplus\mathfrak{sp}(2q,\CC)\oplus\liemc.$$
Here $\mathfrak{sp}(2p,\CC)\oplus\mathfrak{sp}(2q,\CC)$ is the $+1$-eigenspace of $\theta$. It is the  Lie algebra of the complexification $H^\CC$ of the maximal compact subgroup $H$ of $\Sp(2p,2q)$. Of course, $H^\CC=\Sp(2p,\CC)\times\Sp(2q,\CC)$, and also $H=\Sp(2p)\times\Sp(2q)$, the product of the compact symplectic groups, which may be explicitly defined as $$\Sp(2p):=\Sp(2p,\CC)\cap\U(2p)=\{M\in\U(2p)\st \overline{M}J_n=J_nM\},$$ with $J_p:=\begin{pmatrix}
     0 & I_p \\
     -I_p & 0
\end{pmatrix}$.
On the other hand, the $-1$-eigenspace of the Cartan involution $\theta$ is $$\liemc:=\{(B,C)\in\mathrm{M}_{2p\times 2q}(\CC)\times\mathrm{M}_{2q\times 2p}(\CC)\st J_pB=-C^tJ_q\}.$$

Hence, from Definition \ref{definition of Higgs bundle}, we have that an $\Sp(2p,2q)$-Higgs bundle over $X$ is a pair $(E,\varphi)$, where $E$ is a holomorphic $\Sp(2p,\CC)\times\Sp(2q,\CC)$-principal bundle and the Higgs field $\varphi$ is a holomorphic section of $E\times_{\Sp(2p,\CC)\times\Sp(2q,\CC)}\liemc\otimes K$.

If $\mathbb{V}\oplus\mathbb{W}$ is the standard $2p+2q$-dimensional complex representation of $\Sp(2p,\CC)\times\Sp(2q,\CC)$ and $\Omega_{\mathbb{V}}$ and $\Omega_{\mathbb{W}}$ denote the standard symplectic forms on $\mathbb{V}$ and $\mathbb{W}$ respectively, then the isotropy representation space is $$\liemc=\{(b,c)\in\Hom(\mathbb{W},\mathbb{V})\times\Hom(\mathbb{V},\mathbb{W})\,|\,\Omega_{\mathbb{V}}(b\,\cdot,\cdot)=-\Omega_{\mathbb{W}}(\cdot,c\,\cdot)\}.$$
Clearly, if $(b,c)\in\liemc$ then $b$ determines $c$. In fact,
$$\omega^{\mathbb W} c=-b^t\omega^{\mathbb V},$$
where $\omega^{\mathbb V}:V\to V^*$ and $\omega^{\mathbb W}:W\to W^*$ are the isomorphisms induced by $\Omega_{\mathbb{V}}$ and $\Omega_{\mathbb{W}}$.

In the vector bundle language, we have hence the following alternative definition of an $\Sp(2p,2q)$-Higgs bundle.
\begin{definition}
A \emph{$\Sp(2p,2q)$-Higgs bundle} over $X$ is a tuple $(V,\Omega_V,W,\Omega_W,\beta,\gamma)$,
where $(V,\Omega_V)$ and $(W,\Omega_W)$ are holomorphic symplectic vector bundles of rank $2p$ and $2q$ respectively, and
$(\beta,\gamma)\in H^0((\Hom(W,V)\oplus\Hom(V,W))\otimes K)$ are $K$-twisted homomorphisms $\beta:W\to V\otimes K$ and $\gamma:V\to W\otimes K$ such that $\Omega_V(\beta\,\cdot,\cdot)=-\Omega_W(\cdot,\gamma\,\cdot)$. 
\end{definition}

Given an $\Sp(2p,2q)$-Higgs bundle $(V,\Omega_V,W,\Omega_W,\beta,\gamma)$, we must of course have $V\cong V^*$ and $W\cong W^*$, through the skew-symmetric isomorphisms
$$
\omega^V:V\longrightarrow V^*\hspace{1cm}\text{and}\hspace{1cm}\omega^W:W\longrightarrow W^*
$$ induced by $\Omega_V$ and $\Omega_W$, and the condition on $\beta$ and $\gamma$ given on the definition is equivalent to
\begin{equation}\label{eq:condbetagamma}
(\beta^t\otimes 1_K)\omega^V=-(\omega^W\otimes 1_K)\gamma,
\end{equation}
so that $\beta$ determines $\gamma$ (and vice-versa).

For an $\Sp(2p,2q)$-Higgs bundle $(V,\Omega_V,W,\Omega_W,\beta,\gamma)$, we must of course have $$\deg(V)=\deg(W)=0.$$ In other words, the topological invariant of these objects given by the degree is always zero. This is of course consequence of the fact that the group $\Sp(2p,2q)$ is connected and simply-connected and that, for $G$ connected, $G$-Higgs bundles are topologically classified (cf. \cite{ramanathan:1975}) by the elements of $\pi_1G$.

\begin{remark}
 Two $G$-Higgs bundles $(E_{H^\CC},\varphi)$ and $(E_{H^\CC}',\varphi')$ over $X$ are \emph{isomorphic} if there is a holomorphic isomorphism $f:E_{H^\CC}\to E_{H^\CC}'$ such that $\varphi'=\tilde f(\varphi)$, where $\tilde f\otimes 1_K:E(\liemc)\otimes K\to E_{H^\CC}'(\liemc)\otimes K$ is the map induced from $f$ and from the isotropy representation $H^\CC\to\Aut(\liemc)$.
Hence, two $\Sp(2p,2q)$-Higgs bundles $(V,\Omega_V,W,\Omega_W,\beta,\gamma)$ and
$(V',\Omega_{V'},W',\Omega_{W'},\beta')$ are \emph{isomorphic} if there are isomorphisms $f:V\to V'$ and $g:W\to W'$ 
such that $\omega^V=f^t\omega_{V'}f$, $\omega^W=g^t\omega_{W'}g$ and $\beta' g=(f\otimes 1_K)\beta$.
\end{remark}

As $\Sp(2p+2q,\CC)$ is the complexification of $\Sp(2p,2q)$, Higgs bundles for the complex symplectic group will naturally play a role in this paper. Using Definition \ref{definition of Higgs bundle}, and the standard $2n$-dimensional complex representation of $\Sp(2n,\CC)$, we obtain the following definition.

\begin{definition}
A \emph{$\Sp(2n,\CC)$-Higgs bundle} over the compact Riemann surface $X$ is a tuple $((F,\Omega_F),\varphi)$, where $(F,\Omega_F)$ is a holomorphic symplectic vector bundle of rank $2n$ and $\varphi$ is a holomorphic $K$-twisted endomorphism of $F$ which is skew-symmetric with respect to the symplectic form $\Omega_F$ i.e. $\Omega_F(\varphi\,\cdot,\cdot)=-\Omega_F(\cdot,\varphi\,\cdot)$.
\end{definition}

\begin{remark}\label{rem:from SPpq to Spn}
From an $\Sp(2p,2q)$-Higgs bundle $(V,\Omega_V,W,\Omega_W,\beta,\gamma)$, one readily obtains the corresponding $\Sp(2p+2q,\CC)$-Higgs bundle $(E,\varphi)$ by taking $$E:=(V\oplus W,\Omega_V\oplus\Omega_W)$$ and
\begin{equation}\label{eq:eqdefHiggs}
\varphi:=\left(\begin{array}{cc}0 & \beta \\\gamma & 0\end{array}\right)
\end{equation} with respect to the decomposition $V\oplus W$. If $\omega:V\oplus W\to V^*\oplus W^*$ is the isomorphism corresponding to $\Omega_V\oplus\Omega_W$ i.e.
$$\omega=\left(\begin{array}{cc}\omega^V & 0 \\ 0 & \omega^W\end{array}\right),$$
then we must have
$$(\varphi^t\otimes 1_K)\omega=-(\omega\otimes 1_K)\varphi$$
which is obviously equivalent to \eqref{eq:condbetagamma}.
\end{remark}

\section{Moduli spaces}
\subsection{Stability conditions}

We now briefly deduce the stability conditions for $\Sp(2p,2q)$-Higgs bundles. All the details of this theory can be found in \cite{garcia-gothen-mundet:2008}, where several examples are studied.

We begin by stating the stability conditions for $\Sp(2n,\CC)$-Higgs bundles, which will also be needed. The following theorem is proved in \cite[Theorem 4.4]{garcia-gothen-mundet:2008}. Recall that if $(F,\Omega_F)$ is a symplectic
vector bundle, a subbundle $F'\subset F$ is said to be isotropic if the restriction of $\Omega_F$ to $F'$ is identically zero. 

\begin{theorem}\label{thm:sp(2n,C)-stability}
An $\Sp(2n,\CC)$-Higgs
bundle $((F,\Omega_F),\varphi)$ is:
\begin{itemize}
\item\emph{Semistable} if $\deg(F')\leq 0$, for any $\varphi$-invariant, isotropic vector subbundle $F'\subset F$.
\item\emph{Stable} if $\deg(F')<0$, for any $\varphi$-invariant, isotropic, proper vector subbundle $F'\subset F$.
\item\emph{Polystable} if it is semistable and for each $\varphi$-invariant, isotropic, proper vector subbundle $F'\subset F$ such that $\deg F'=0$ there is another coisotropic vector subbundle
$F''\subset F$ which is $\varphi$-invariant and $F=F'\oplus F''$.
\end{itemize} 
\end{theorem}

In order to state the stability condition for $\Sp(2p,2q)$-Higgs bundles, we first introduce some notation.
For each vector subbundle $V'$ of $V$, denote by $V'^{\perp_{\Omega_V}}$ the orthogonal complement of $V'$ with respect to the symplectic form $\Omega_V$. Define similarly $W'^{\perp_{\Omega_W}}$ for a subbundle $W'\subset W$.

For any pair of filtrations
\begin{equation}\label{eq:filtration V}
\mathcal V:=(0=V_0\subsetneq V_1\subsetneq V_2\subsetneq\dots\subsetneq V_k=V)
\end{equation}
\begin{equation}\label{eq:filtration W}
\mathcal W:=(0=W_0\subsetneq W_1\subsetneq W_2\subsetneq\dots\subsetneq W_l=W)
\end{equation}
satisfying $V_{k-i} = V_i^{\perp_{\Omega_V}}$ and $W_{l-j} = W_j^{\perp_{\Omega_W}}$, let
$$\Lambda(\mathcal V):=\{(\lambda_1,\lambda_2,\dots,\lambda_k)\in\RR^k
\mid \lambda_i\leq \lambda_{i+1}\text{ and
}\lambda_{k-i+1}=-\lambda_i\text{ for any }i\},$$
$$\Lambda(\mathcal W):=\{(\mu_1,\mu_2,\dots,\mu_l)\in\RR^l
\mid \mu_j\leq \mu_{j+1}\text{ and
}\mu_{l-j+1}=-\mu_j\text{ for any }j\}.$$ For each $\lambda=(\lambda_1,\lambda_2,\dots,\lambda_k)\in\Lambda(\mathcal V)$ and $\mu=(\mu_1,\mu_2,\dots,\mu_l)\in\Lambda(\mathcal W)$, consider the subbundle of $\End(V\oplus W)\otimes K$ defined by
$$\tilde N(\mathcal V,\mathcal W,\lambda,\mu):=\bigcap_{\lambda_i=\mu_j}\left\{(b,c)\in\End(V\oplus W)\otimes K\st b(W_j)\subseteq V_i\otimes K,\, c(V_i)\subseteq(W_j)\otimes K\right\}$$
and let
$$
d(\mathcal V,\lambda):=\sum_{i=1}^{k-1}(\lambda_i-\lambda_{i+1})\deg V_i\hspace{.5cm}\text{and}\hspace{.5cm}d(\mathcal W,\mu):=\sum_{j=1}^{l-1}(\mu_j-\mu_{j+1})\deg W_j.$$

Notice that, for an $\Sp(2p,2q)$-Higgs bundle $(V,\Omega_V,W,\Omega_W,\beta,\gamma)$, if $V'\subset V$ and $W'\subset W$, then \eqref{eq:condbetagamma} implies the following equivalence:
\begin{equation}\label{eq:betainvandgammainv}
\beta(W')\subset V'\otimes K\Longleftrightarrow\gamma(V')\subset W'\otimes K.
\end{equation}
Both conditions are clearly equivalent to the $\varphi$-invariance of $V'\oplus W'\subset V\oplus W$ where $\varphi$ is given by \eqref{eq:eqdefHiggs}.
Another way to state equivalence \eqref{eq:betainvandgammainv} is the following:
$$(\beta,\gamma)\in H^0(\tilde N(\mathcal V,\mathcal W,\lambda,\mu))\Longleftrightarrow \beta\in H^0( N(\mathcal V,\mathcal W,\lambda,\mu))$$ where
$$N(\mathcal V,\mathcal W,\lambda,\mu):=\bigcap_{\lambda_i=\mu_j}\left\{b\in\Hom(W,V)\otimes K\st b(W_j)\subseteq V_i\otimes K\right\}\subseteq \Hom(W,V)\otimes K.$$

Having these definitions, and according to \cite{garcia-gothen-mundet:2008}, we can now state the stability conditions for an $\Sp(2p,2q)$-Higgs bundle.
\begin{proposition}\label{prop:sp(2p,2q)-alpha-stability}
An $\Sp(2p,2q)$-Higgs bundle
$(V,\Omega_V,W,\Omega_W,\beta,\gamma)$ is:
\begin{itemize}
 \item\emph{Semistable} if and only if $d(\mathcal V,\lambda)+d(\mathcal W,\mu)\geq 0$ for
  any choice of filtrations $\mathcal V$ and $\mathcal W$ as in \eqref{eq:filtration V} and \eqref{eq:filtration W} and any
  $(\lambda,\mu)\in\Lambda(\mathcal V)\times\Lambda(\mathcal W)$ such that
  $\beta\in H^0(N(\mathcal V,\mathcal W,\lambda,\mu))$.
 \item\emph{Stable} if and only if $d(\mathcal V,\lambda)+d(\mathcal W,\mu)> 0$ for
  any choice of filtrations $\mathcal V$ and $\mathcal W$ as in \eqref{eq:filtration V} and \eqref{eq:filtration W} and any
  $(\lambda,\mu)\in\Lambda(\mathcal V)\times\Lambda(\mathcal W)\setminus\{(0,0)\}$ such that
  $\beta\in H^0(N(\mathcal V,\mathcal W,\lambda,\mu))$.
\item\emph{Polystable} if and only if it is
 semistable and, for any choice of filtrations $\mathcal V$ and $\mathcal W$ as in \eqref{eq:filtration V} and \eqref{eq:filtration W} and any
 $(\lambda,\mu)\in\Lambda(\mathcal V)\times\Lambda(\mathcal W)$ satisfying $\lambda_i<\lambda_{i+1}$ and $\mu_i<\mu_{i+1}$
 for each $i$, such that $\beta\in H^0(N(\mathcal V,\mathcal W,\lambda,\mu))$ and
 $d(\mathcal V,\lambda)+d(\mathcal W,\mu)=0$, there are isomorphisms
 $$V\simeq V_1\oplus V_2/V_1\oplus\dots\oplus V_k/V_{k-1}$$
 and
 $$W\simeq W_1\oplus W_2/W_1\oplus\dots\oplus W_l/W_{l-1}$$
 such that $$\Omega_V(V_i/V_{i-1},V_j/V_{j-1})=0,\ \text{ unless }\ j=k+1-i$$ and $$\Omega_W(W_i/W_{i-1},W_j/W_{j-1})=0,\ \text{ unless }\ j=l+1-i.$$
Moreover, via this isomorphism,
 $$\beta\in H^0\bigg(\bigoplus_{\lambda_j=\mu_i}\Hom(W_j/W_{j-1},V_i/V_{i-1})\otimes K\bigg).$$
\end{itemize}
\end{proposition}

There is a simplification of the stability condition for $\Sp(2p,2q)$-Higgs bundles analogous to the cases considered in \cite{garcia-gothen-mundet:2008}. 

Bearing \eqref{eq:betainvandgammainv} in mind, we can now state the simplified version of the (semi,poly)stability condition for $\Sp(2p,2q)$-Higgs bundles.

\begin{theorem}\label{thm:orthogonal-stability}
  An $\Sp(2p,2q)$-Higgs bundle $(V,\Omega_V,W,\Omega_W,\beta,\gamma)$
  is:
\begin{itemize}
 \item Semistable if and only if $\deg V'+\deg W'\leq 0$ for any pair of isotropic subbundles $V'\subset V$ and $W'\subset W$ such that $\beta(W')\subset V'\otimes K$.
 \item Stable if and only if $\deg V'+\deg W'<0$ for any pair of isotropic subbundles $V'\subset V$ and $W'\subset W$ such that at least one of them is a proper subbundle and $\beta(W')\subset V'\otimes K$.
 \item Polystable if and only if it is semistable and, for any pair of isotropic (resp. coisotropic) subbundles $V'\subset V$ and $W'\subset W$, for which at least one of them is a proper subbundle and $\beta(W')\subset V'\otimes K$, such that $\deg V'+\deg W'=0$, there are other coisotropic (resp. isotropic) subbundles $V''\subset V$ and $W''\subset W$ with $\beta(W'')\subset V''\otimes K$  so that $V\cong V'\oplus V''$ and $W\cong W'\oplus W''$.
\end{itemize}
\end{theorem}

\begin{proof}
Let us deal first with the semistability statement.
Consider an $\Sp(2p,2q)$-Higgs bundle $(V,\Omega_V,W,\Omega_W,\beta,\gamma)$ for which the stated condition holds: for any pair of isotropic subbundles $V'\subset V$ and $W'\subset W$ such that $\beta(W')\subset V'\otimes K$, we have $$\deg V'+\deg W'\leq 0.$$
 We want to prove that $(V,\Omega_V,W,\Omega_W,\beta,\gamma)$
is semistable and we will do it by making use of Proposition \ref{prop:sp(2p,2q)-alpha-stability}. Suppose that $\beta$ is nonzero, for otherwise the result follows from the usual characterization of (semi)stability for $\Sp(2p,\CC)\times\Sp(2q,\CC)$-principal bundles due to Ramanathan (see \cite[Remark 3.1]{ramanathan:1975}).

Choose any pair of filtrations as in \eqref{eq:filtration V} and \eqref{eq:filtration W},
satisfying $V_{k-i} = V_i^{\perp_{\Omega_V}}$ and $W_{l-j} = W_j^{\perp_{\Omega_W}}$. As before, let
$$\Lambda(\mathcal V):=\{\lambda=(\lambda_1,\lambda_2,\dots,\lambda_k)\in\RR^k
\mid \lambda_i\leq \lambda_{i+1}\text{ and
}\lambda_{k-i+1}=-\lambda_i\text{ for any }i\},$$
$$\Lambda(\mathcal W):=\{\mu=(\mu_1,\mu_2,\dots,\mu_l)\in\RR^l
\mid \mu_j\leq \mu_{j+1}\text{ and
}\mu_{l-j+1}=-\mu_j\text{ for any }j\},$$
and consider the convex set
$$\Lambda(\mathcal V,\mathcal W,\beta):= \{(\lambda,\mu)\in\Lambda(\mathcal V)\times\Lambda(\mathcal W)\mid \beta\in H^0(N(\mathcal V,\mathcal W,\lambda,\mu))\}\subset\RR^k\times\RR^l,$$
where $$N(\mathcal V,\mathcal W,\lambda,\mu):=\bigcap_{\lambda_i=\mu_j}\left\{b\in\Hom(W,V)\otimes K\st b(W_j)\subseteq V_i\otimes K\right\}.$$

Define
$
\mathcal J:=\{(i,j)\mid
\beta(W_j)\subset V_i\otimes K\}=\{(i_1,j_1),\dots,(i_r,j_r)\},$ with $r\leq\min\{k,l\}$.
It is easily checked that if $\lambda=(\lambda_1,\dots,\lambda_k)\in\Lambda(\mathcal V)$ and $\mu=(\mu_1,\dots,\mu_l)\in\Lambda(\mathcal W)$ then, for $(i_m,j_m),(i_{m+1},j_{m+1})\in\mathcal J$ the following holds:
\begin{equation}\label{eq:Lambda-JJJ}
(\lambda,\mu)\in\Lambda(\mathcal V, \mathcal W,\beta)
\Longleftrightarrow \lambda_a=\lambda_{i_{m+1}}=\mu_{j_{m+1}}=\mu_b, 
\end{equation}
for every $i_m<a\leq i_{m+1}$ and $j_m<b\leq j_{m+1}$. Also, the relations $\Omega_V(\beta\,\cdot,\cdot)=-\Omega_W(\cdot,\gamma\,\cdot)$, $V_{k-i} = V_i^{\perp_{\Omega_V}}$ and $W_{l-j} = W_j^{\perp_{\Omega_W}}$ imply that the set of indices in $\mathcal J$ is symmetric:
\begin{equation}\label{eq:JJJ-symmetric}
(i,j)\in\mathcal J \Longleftrightarrow (k-i,l-j)\in\mathcal J.
\end{equation}

Let now $\mathcal J':=\{(i,j)\in\mathcal J\mid 2i\leq k, 2j\leq l\}$ and, for each
$(i,j)\in\mathcal J'$, define the vectors
$$L_i:=-\sum_{c\leq i}e_c+\sum_{d\geq k-i+1}e_d\text{ and }M_j:=-\sum_{c\leq j}e'_c+\sum_{d\geq l-j+1}e'_d$$
where $e_1,\dots,e_k$ and $e'_1,\dots,e'_l$ are the canonical basis of $\RR^k$ and of $\RR^l$ respectively. From (\ref{eq:Lambda-JJJ}) and (\ref{eq:JJJ-symmetric}), we conclude
that $\Lambda(\mathcal V, \mathcal W,\beta)$ is the positive span of the set $\{
L_i,M_j\mid (i,j)\in\mathcal J'\}$. Hence,
$$d(\mathcal V,\lambda)+d(\mathcal W,\mu)\geq 0\text{ for any }(\lambda,\mu)\in\Lambda(\mathcal V,\mathcal W,\beta)$$ if and only if $$d(\mathcal V,L_i)+d(\mathcal W,M_j)\geq 0\text{ for any }(i,j)\in\mathcal{J}'.$$
Now, we compute $$d(\mathcal V,L_i)+d(\mathcal W,M_j)=-\deg V_{k-i}-\deg V_i-\deg W_{l-j}-\deg W_l=-2(\deg V_i+\deg W_j)$$ where in the second equality we have used that $\deg V_{k-i}=\deg V_i$ and $\deg W_{l-j}=\deg W_l$, since $V_{k-i}=V_i^{\perp_{\Omega_V}}$ and $W_{l-j}=W_j^{\perp_{\Omega_W}}$. Therefore $d(\mathcal V,\lambda)+d(\mathcal W,\mu)\geq 0$ for any $(\lambda,\mu)\in\Lambda(\mathcal V,\mathcal W,\beta)$ is
equivalent to $\deg V_i+\deg W_j\leq 0$ for every $(i,j)\in\mathcal J'$, which holds by assumption, because, for such $(i,j)$, we have $\beta(W_j)\subset V_i\otimes K$ and since $W_j$ and $V_i$ are isotropic. Hence, from Proposition \ref{prop:sp(2p,2q)-alpha-stability}, it follows that $(V,\Omega_V,W,\Omega_W,\beta,\gamma)$ is semistable.

The converse statement is readily obtained by applying the semistability condition of Proposition \ref{prop:sp(2p,2q)-alpha-stability} to the filtrations
 $0\subset V'\subset V'^{\perp_{\Omega}}\subset V$ and $0\subset W'\subset W'^{\perp_{\Omega}}\subset W$.

The proof of the second and third items follow along the same lines.
\end{proof}

\begin{remark}
Given an $\Sp(2p,2q)$-Higgs bundle $(V,\Omega_V,W,\Omega_W,\beta,\gamma)$, we saw in Remark \ref{rem:from SPpq to Spn}, that we can construct a corresponding $\Sp(2p+2q,\CC)$-Higgs bundle $(V\oplus W, \Omega_V\oplus\Omega_W,\varphi)$, with $$\varphi:=\left(\begin{array}{cc}0 & \beta \\\gamma & 0\end{array}\right).$$ It will be shown in Section \ref{sec:stable and non-simple} that $(V,\Omega_V,W,\Omega_W,\beta,\gamma)$ is stable (as in Theorem \ref{thm:orthogonal-stability}) if and only if $(V\oplus W, \Omega_V\oplus\Omega_W,\varphi)$ is stable (as in Theorem \ref{thm:sp(2n,C)-stability}). 
\end{remark}

Through the above theorem, one can define the moduli space $\M_{\Sp(2p,2q)}$ of polystable $\Sp(2p,2q)$-Higgs bundles. The construction of the moduli spaces of $G$-Higgs bundles is a particular case of a general
construction of Schmitt  \cite{schmitt:2008}, using methods of Geometric Invariant Theory, and showing that they carry a natural structure of complex algebraic variety.

\subsection{Deformation theory of $\Sp(2p,2q)$-Higgs bundles}

In this section, we briefly study the deformation theory of $\Sp(2p,2q)$-Higgs bundles and, in particular, the identification of the tangent space of $\M_{\Sp(2p,2q)}$ at the smooth points with the first hypercohomology group of a certain complex of sheaves over the Riemann surface $X$. All basic notions can be found in detail in \cite{garcia-gothen-mundet:2008}.

For each $\Sp(2p,2q)$-Higgs bundle $(V,\Omega_V,W,\Omega_W,\beta,\gamma)$, let
\begin{equation}\label{eq:M}
M(V,\Omega_V,W,\Omega_W):=\{(f,g)\in\Hom(W,V)\oplus\Hom(V,W)\st\omega^Wg=-f^t\omega^V\}\cong\Hom(W,V).
\end{equation} Then, there is a  complex defined as
\begin{equation}\label{eq:defcomplex}
C^\bullet:=C^\bullet_{(V,\Omega_V,W,\Omega_W,\beta,\gamma)}:\Lambda_{\Omega_V}^2V\oplus\Lambda_{\Omega_W}^2W\xrightarrow{\ad(\beta,\gamma)}M(V,\Omega_V,W,\Omega_W)\otimes K
\end{equation}
where $\Lambda_{\Omega_V}^2V$ (resp. $\Lambda_{\Omega_W}^2W$) denotes the bundle of endomorphisms of $V$ (resp. $W$) which are skew-symmetric with respect to $\Omega_V$ (resp. $\Omega_W$), and where $$\ad(\beta,\gamma)(f,g)=(\beta g-(f\otimes 1_K)\beta,\gamma f-(g\otimes 1_K)\gamma).$$
This is in fact a particular case of a general complex for every $G$-Higgs bundle $(E_{H^\CC},\varphi)$, given by $C^\bullet_{(E_{H^\CC},\varphi)}:E_{H^\CC}(\liehc)\xrightarrow{d\iota(\varphi)}E_{H^\CC}(\liemc)\otimes K$, where $d\iota$ represents the differential of the isotropy representation.

The following proposition follows from the general theory of $G$-Higgs bundles. The first item generalises the well-known fact that the deformation space of a holomorphic vector bundle $V$ is given by $H^1(\End(V))$ and it can be proved for instance by using \v{C}ech cohomology to represent an infinitesimal deformation as an element of $\mathrm{Spec}(\CC[\epsilon]/(\epsilon^2))$. The second item follows from the property of hypercohomology which states the existence of a long exact sequence associated to a short exact sequence of complexes. A convenient reference for these facts is \cite{biswas-ramanan:1994}.

\begin{proposition}\label{deformation complex for Sp(2p,2q)}
Let $(V,\Omega_V,W,\Omega_W,\beta,\gamma)$ be an $\Sp(2p,2q)$-Higgs bundle over $X$.
\begin{enumerate}
    \item The
infinitesimal deformation space of $(V,\Omega_V,W,\Omega_W,\beta,\gamma)$ is isomorphic to
the first hypercohomology group
$\HH^1(C^\bullet)$ of the complex \eqref{eq:defcomplex}.

\noindent In particular, if $(V,\Omega_V,W,\Omega_W,\beta,\gamma)$
represents a smooth point of $\M_{\Sp(2p,2q)}$, then the tangent space of $\M_{\Sp(2p,2q)}$ at this point is canonically isomorphic to $\HH^1(C^\bullet)$.
    \item There is an exact sequence
\begin{equation*}
\begin{split}
0&\longrightarrow\HH^0(C^\bullet)\longrightarrow
H^0(\Lambda_{\Omega_V}^2V\oplus\Lambda_{\Omega_W}^2W)\longrightarrow H^0(M(V,\Omega_V,W,\Omega_W)\otimes K)\longrightarrow\\
&\longrightarrow\HH^1(C^\bullet)\longrightarrow H^1(\Lambda_{\Omega_V}^2V\oplus\Lambda_{\Omega_W}^2W)\longrightarrow
H^1(M(V,\Omega_V,W,\Omega_W)\otimes K)\longrightarrow\\
&\longrightarrow\HH^2(C^\bullet)\longrightarrow 0
\end{split}
\end{equation*}
where
the maps $H^i(\Lambda_{\Omega_V}^2V\oplus\Lambda_{\Omega_W}^2W)\to
H^i(M(V,\Omega_V,W,\Omega_W)\otimes K)$ are induced by $\ad(\beta,\gamma)$ and $M(V,\Omega_V,W,\Omega_W)$ is defined in \eqref{eq:M}.
\end{enumerate}
\end{proposition}

\subsection{Stable and non-simple $\Sp(2p,2q)$-Higgs bundles}\label{sec:stable and non-simple}

An automorphism of an $\Sp(2p,2q)$-Higgs bundle $(V,\Omega_V,W,\Omega_W,\beta,\gamma)$ is a pair $(f,g)$ of automorphisms of the symplectic vector bundles $(V,\Omega_V)$ and $(W,\Omega_W)$ which preserve $\beta$ and $\gamma$. In other words, $(f,g)\in\Aut(V)\times\Aut(W)$ must be such that $$\omega^V f=(f^t)^{-1}\omega^V,\hspace{.5cm}\omega^W g=(g^t)^{-1}\omega^W$$ and $$f\beta=\beta g.$$ (which is equivalent to $g\gamma=\gamma f$). Thus,
$$\Aut(V,\Omega_V,W,\Omega_W,\beta,\gamma)=\{(f,g)\in\Aut(V,\Omega_V)\times\Aut(W,\Omega_W)\st f\beta=\beta g\}.$$

\begin{definition}
An $\Sp(2p,2q)$-Higgs bundle $(V,\Omega_V,W,\Omega_W,\beta,\gamma)$ is said to be \emph{simple} if $\Aut(V,\Omega_V,W,\Omega_W,\beta,\gamma)=\pm (1_V, 1_W)\cong\Z/2$. 
\end{definition}

So, an $\Sp(2p,2q)$-Higgs bundle is simple if it admits the minimum possible automorphisms. In contrast to the case of vector bundles, stable $G$-Higgs bundles may not be simple \cite{garcia-gothen-mundet:2008 II,garcia-oliveira:2012}.
Our purpose in this section is to give an explicit description of $\Sp(2p,2q)$-Higgs bundles which are stable but not simple.

As we already saw in Remark \ref{rem:from SPpq to Spn}, from an $\Sp(2p,2q)$-Higgs bundle $(V,\Omega_V,W,\Omega_W,\beta,\gamma)$, one constructs an associated $\Sp(2p+2q,\CC)$-Higgs bundle by taking $(V\oplus W, \Omega_V\oplus\Omega_W,\varphi)$, with $$\varphi:=\left(\begin{array}{cc}0 & \beta \\\gamma & 0\end{array}\right).$$

\begin{proposition}\label{prop:equivalence of semistability}
 Let $(V,\Omega_V,W,\Omega_W,\beta,\gamma)$ be an $\Sp(2p,2q)$-Higgs bundle and $(V\oplus W, \Omega_V\oplus\Omega_W,\varphi)$ be the corresponding $\Sp(2p+2q,\CC)$-Higgs bundle. Then $(V,\Omega_V,W,\Omega_W,\beta,\gamma)$ is stable if and only if $(V\oplus W, \Omega_V\oplus\Omega_W,\varphi)$ is stable.
\end{proposition}
\proof
If $(V\oplus W, \Omega_V\oplus\Omega_W,\varphi)$ is stable, it is clear that $(V,\Omega_V,W,\Omega_W,\beta,\gamma)$ is stable.
Suppose that $(V,\Omega_V,W,\Omega_W,\beta,\gamma)$ is stable, and let $U\subset V\oplus W$ be a $\varphi$-invariant isotropic subbundle. Let $U'$ be the kernel of the projection $V\oplus W\to W$ restricted to $U$. This is a vector subbundle of $U$ and of $V$ because $X$ is a compact Riemann surface. Denote the quotient vector bundle by $U''=U/U'$. It is a vector subbundle of $W$. In other words, we have the following commutative diagram: 
$$
\xymatrix{
&0\ar[r]&U'\ar@{^{(}->}[d]\ar[r]&U\ar@{^{(}->}[d]\ar[r]&U''\ar@{^{(}->}[d]\ar[r]&0\\
&0\ar[r]&V\ar[r]&V\oplus W\ar[r]&W\ar[r]&0.}
$$
Since $U\subset U^{\perp_{\Omega_V\oplus\Omega_W}}$, it follows that $U'\subset U'^{\perp_{\Omega_V}}$ and $U''\subset U''^{\perp_{\Omega_W}}$ i.e. $U'\subset V$ and $U''\subset W$ are both isotropic. On the other hand, the $\varphi$-invariance of $U$ implies that $\beta(U'')\subset U'\otimes K$ and $\gamma(U')\subset U''\otimes K$. The stability of $(V,\Omega_V,W,\Omega_W,\beta,\gamma)$ implies then $\deg(U')+\deg(U'')<0$, i.e. $\deg(U)<0$.
\endproof

Since the direct sum of stable mutually non-isomorphic $\Sp(2n,\CC)$-Higgs bundles is stable (the proof of this fact for $\Sp(2n,\CC)$-bundles is given in \cite[Lemma 2.1]{hitching:2005}, and the proof ready adapts to the Higgs bundle situation), we have the following immediate corollary.

\begin{corollary}\label{cor:direct sum of stable}
An $\Sp(2p,2q)$-Higgs bundle is stable if and only if it is a direct sum of stable, mutually non-isomorphic, $\Sp(2p',2q')$-Higgs bundles.
\end{corollary}

The description of stable and non-simple $\Sp(2p,2q)$-Higgs bundles is now straightforward.

\begin{proposition}\label{prop:description of stable and non-simple}
 An $\Sp(2p,2q)$-Higgs bundle is stable and non-simple if and only if it decomposes as a direct sum of stable and simple $\Sp(2p_i,2q_i)$-Higgs bundles. In other words, an $\Sp(2p,2q)$-Higgs bundle $(V,\Omega_V,W,\Omega_W,\beta,\gamma)$ is stable and non-simple if and only if 
 \begin{equation}\label{eq:direct sum}
 (V,\Omega_V,W,\Omega_W,\beta,\gamma)=\bigoplus_{i=1}^r(V_i,\Omega_{V_i},W_i,\Omega_{W_i},\beta_i)
 \end{equation} where $(V_i,\Omega_{V_i},W_i,\Omega_{W_i},\beta_i)$ are stable and simple $\Sp(2p_i,2q_i)$-Higgs bundles, $r>1$ and $(V_i,\Omega_{V_i},W_i,\Omega_{W_i},\beta_i)\ncong (V_j,\Omega_{V_j},W_j,\Omega_{W_j},\beta_j)$ if $i\neq j$.
\end{proposition}
\proof
Consider an $\Sp(2p,2q)$-Higgs bundle given by \eqref{eq:direct sum}. Since each summand is stable and since the summands are not isomorphic to each other, from Corollary \ref{cor:direct sum of stable}, we see that it is stable. Moreover, each summand is simple, so $\Aut(V,\Omega_V,W,\Omega_W,\beta,\gamma)\cong(\Z/2)^r$.
\endproof

One can easily check that an exact analogue situation occurs for $\Sp(2n,\CC)$-Higgs bundles (the case for symplectic bundles is proved in \cite[Corollary 2.2]{hitching:2005}). Thus we have the following corollary.

\begin{corollary}\label{cor:equivalence of stable and simple}
 Let $(V,\Omega_V,W,\Omega_W,\beta,\gamma)$ be an $\Sp(2p,2q)$-Higgs bundle and $(V\oplus W, \Omega_V\oplus\Omega_W,\varphi)$ be the corresponding $\Sp(2p+2q,\CC)$-Higgs bundle. Then $(V,\Omega_V,W,\Omega_W,\beta,\gamma)$ is stable and simple if and only if $(V\oplus W, \Omega_V\oplus\Omega_W,\varphi)$ is stable and simple.
\end{corollary}

Recall that $\M_{\Sp(2p,2q)}$ denotes the moduli space of polystable $\Sp(2p,2q)$-Higgs bundles.
The following result will be important below. Using \cite[Proposition 3.18]{garcia-gothen-mundet:2008}, it is straightforward from the fact that the complexification of $\Sp(2p,2q)$ is $\Sp(2p+2q,\CC)$ and from Corollary \ref{cor:equivalence of stable and simple}:

\begin{proposition}\label{prop:stable and simple imply smooth}
A stable and simple $\Sp(2p,2q)$-Higgs bundle corresponds to a smooth point of the moduli space $\M_{\Sp(2p,2q)}$.
\end{proposition}

So, from \cite{schmitt:2008}, at a point of $\M_{\Sp(2p,2q)}$ represented by a stable and simple $(V,\Omega_V,W,\Omega_W,\beta,\gamma)$, there exists a local universal family, hence, from Proposition \ref{deformation complex for Sp(2p,2q)}, the dimension of the component of $\M_{\Sp(2p,2q)}$ containing that point is the expected dimension given by $\dim\mathbb{H}^1(C^\bullet)$,
where $C^\bullet$ is the complex given by \eqref{eq:defcomplex}.
Since Corollary \ref{cor:equivalence of stable and simple} says that the corresponding $\Sp(2p+2q,\CC)$-Higgs bundle $(V\oplus W,\Omega_V\oplus\Omega_W,\varphi)$ is stable and simple then $\Aut(V\oplus W,\Omega_V\oplus\Omega_W,\varphi)\cong\Z/2$ also. So, if 
$$C_{\Sp(2p+2q,\CC)}^\bullet:\Lambda_{\Omega_V\oplus\Omega_W}^2(V\oplus W)\xrightarrow{\ad(\varphi)}\Lambda_{\Omega_V\oplus\Omega_W}^2(V\oplus W)\otimes K$$ is the deformation complex of the $\Sp(2p+2q,\CC)$-Higgs bundle $(V\oplus W,\Omega_V\oplus\Omega_W,\varphi)$, then
using \cite[Corollary 3.16]{garcia-gothen-mundet:2008}, we know that
$0=\mathbb{H}^0(C_{\Sp(2p+2q,\CC)}^\bullet)=\mathbb{H}^0(C^\bullet)\oplus\mathbb{H}^2(C^\bullet)^*$.
Therefore 
$\dim\mathbb{H}^2(C^\bullet(V,\Omega_V,W,\Omega_W,\beta,\gamma))=0$. So (ii) of Proposition \ref{deformation complex for Sp(2p,2q)} yields $$\dim\M_{\Sp(2p,2q)}=\dim\mathbb{H}^1(C^\bullet)=(p(2p+1)+q(2q+1)+4pq)(g-1).$$

\subsection{Polystable $\Sp(2p,2q)$-Higgs bundles}

Now we look at polystable $\Sp(2p,2q)$-Higgs bundles. First notice that we can realize $\U(p,q)$ as a subgroup of $\Sp(2p,2q)$, using the injection $$A\mapsto \begin{pmatrix}
A & 0\\
0 & \overline A\end{pmatrix}.$$ 
Again from Definition \ref{definition of Higgs bundle}, we obtain the following definition of $\U(p,q)$-Higgs bundles on the Riemann surface $X$:
\begin{definition}
A \emph{$\U(p,q)$-Higgs bundle} on $X$ is a quadruple $(V',W',\beta',\gamma')$ where $V'$ and $W'$ are vector bundles of rank $p$ and $q$, and the Higgs field is given by $(\beta',\gamma')$, where $\beta'\in H^0(\Hom(W',V')\otimes K)$ and $\gamma'\in H^0(\Hom(V',W')\otimes K)$.
\end{definition}

From a $\U(p,q)$-Higgs bundle $(V',W',\beta',\gamma')$ and from the inclusion of $\U(p,q)$ in $\Sp(2p,2q)$, one obtains the corresponding $\Sp(2p,2q)$-Higgs bundle 
$$(V'\oplus V'^*,\langle\, , \rangle_{V'},W'\oplus W'^*,\langle\, , \rangle_{W'},\beta'-\gamma'^t,\gamma'-\beta'^t),$$ where $\langle\, , \rangle_{V'}$ and $\langle\, , \rangle_{W'}$ denote the standard symplectic forms on $V'\oplus V'^*$ and on $W'\oplus W'^*$.

\begin{theorem}\label{thm:description of polystable}
 Let $(V,\Omega_V,W,\Omega_W,\beta,\gamma)$ be a polystable $\Sp(2p,2q)$-Higgs bundle. There is a decomposition of $(V,\Omega_V,W,\Omega_W,\beta,\gamma)$ as a sum of stable $G_\alpha$-Higgs bundles, where each $G_\alpha$ is one of the following subgroups of $\Sp(2p,2q)$: $\Sp(2p_\alpha,2q_\alpha)$, $\U(p_\alpha,q_\alpha)$, $\Sp(2n_\alpha)$ or $\U(n_\alpha)$, where $p_\alpha\leq p$, $q_\alpha\leq q$ and $n_\alpha\leq p+q$.
\end{theorem}
\proof
Since $(V,\Omega_V,W,\Omega_W,\beta,\gamma)$ is polystable, we know, from Proposition \ref{prop:sp(2p,2q)-alpha-stability}, that it is semistable and that the following holds. Take any filtrations $$\mathcal V=(0=V_0\subsetneq V_1\subsetneq V_2\subsetneq\dots\subsetneq V_k=V)$$ such that
$$\mathcal W=(0=W_0\subsetneq W_1\subsetneq W_2\subsetneq\dots\subsetneq W_l=W)$$
satisfying $V_{k-i} = V_i^{\perp_{\Omega_V}}$ and $W_{l-j} = W_j^{\perp_{\Omega_W}}$; take
any  $(\lambda,\mu)\in\Lambda(\mathcal V)\times\Lambda(\mathcal W)$ for which $\lambda_i<\lambda_{i+1}$ and $\mu_i<\mu_{i+1}$
 for every $i$. Suppose furthermore that $\beta\in H^0(N(\mathcal V,\mathcal W,\lambda,\mu))$ and
 $d(\mathcal V,\lambda)+d(\mathcal W,\mu)=0$. Then there are isomorphisms
 \begin{equation}\label{eq:splitting}
 \begin{split}
 V&\simeq V_1\oplus V_2/V_1\oplus\dots\oplus V_k/V_{k-1}\\
 W&\simeq W_1\oplus W_2/W_1\oplus\dots\oplus W_l/W_{l-1}
 \end{split}
 \end{equation} 
 such that 
  \begin{equation}\label{eq:symplectic form}
 \begin{split}
\Omega_V(V_i/V_{i-1},V_j/V_{j-1})&=0,\ \text{ unless }\ j=k+1-i\\
\Omega_W(W_i/W_{i-1},W_j/W_{j-1})&=0,\ \text{ unless }\ j=l+1-i
 \end{split}
 \end{equation} 
and that, via this isomorphism,
\begin{equation}\label{eq:Higgs field splitted}
\beta\in H^0\bigg(\bigoplus_{\lambda_i=\mu_j}\Hom(W_j/W_{j-1},V_i/V_{i-1})\otimes K\bigg).
\end{equation}

Now we analyse the possible cases. Condition \eqref{eq:symplectic form} tells us that, with respect to decomposition \eqref{eq:splitting}, we have
$$ \omega^V=\begin{pmatrix}
    0 & 0 & 0 & \dots & -(\omega^V_1)^t \\
    \vdots &  \dots & 0 & \dots & \vdots \\
    0 & 0 &  \dots &  \dots & 0\\
    0 & \omega^V_2 & 0 & \dots & 0 \\
    \omega^V_1 & 0 & 0 & \dots & 0\
  \end{pmatrix},\hspace{.5cm}
  \omega^W=\begin{pmatrix}
    0 & 0 & 0 & \dots & -(\omega^W_1)^t \\
    \vdots &  \dots & 0 & \dots & \vdots \\
    0 & 0 &  \dots &  \dots & 0\\
    0 & \omega^W_2 & 0 & \dots & 0 \\
    \omega^W_1 & 0 & 0 & \dots & 0\
  \end{pmatrix}$$
where $\omega^V_i:V_i/V_{i-1}\stackrel{\cong}{\longrightarrow}(V_{k+1-i}/V_{k-i})^*$ is the isomorphism induced by $\Omega_V$ and similarly for $\omega^W_i$. 

For any $k,l$, whenever we have $\lambda_i=\mu_j$ for some $i\neq \frac{k+1}{2}$ and $j\neq\frac{l+1}{2}$,  we also have $\lambda_{k+1-i}=\mu_{l+1-j}$. The symplectic forms do not restrict to $V_i/V_{i-1}$ and to $W_j/W_{j-1}$, but we have the restrictions $$\beta_j:W_j/W_{j-1}\longrightarrow V_i/V_{i-1}\otimes K,\hspace{.3cm}\beta_{l+1-j}:W_{l+1-j}/W_{l-j}\longrightarrow V_{k+1-i}/V_{k-i}\otimes K$$ of $\beta$ to  $W_j/W_{j-1}$ and $W_{l+1-j}/W_{l-j}$ respectively. If $$\gamma_i:V_i/V_{i-1}\longrightarrow W_j/W_{j-1}\otimes K$$ is given by $\gamma_i=-((\omega_j^W)^{-1}\otimes 1_K)(\beta_{l+1-j}\otimes 1_K)\omega_i^V$ and if $p'=\rk(V_i/V_{i-1})$ and $q'= \rk(W_j/W_{j-1})$, then $$(V_i/V_{i-1},W_j/W_{j-1},\beta_j,\gamma_i)$$ is a $\U(p',q')$-Higgs bundle. Of course one also obtains the corresponding dual $\U(p',q')$-Higgs bundle, being a $\U(p'+q')$-Higgs bundle if $\beta_j=\gamma_i=0$.

If $\lambda_{\frac{k+1}{2}}=\mu_{\frac{l+1}{2}}$, then we must have $k,l$ both odd, and indeed $\lambda_{\frac{k+1}{2}}=\mu_{\frac{l+1}{2}}=0$ so, from \eqref{eq:Higgs field splitted} $\beta$ restricts to $\beta_{\frac{l+1}{2}}:W_{\frac{l+1}{2}}/W_{\frac{l-1}{2}}\to V_{\frac{k+1}{2}}/V_{\frac{k-1}{2}}$. Also, from \eqref{eq:symplectic form}, $\Omega_V$ and $\Omega_W$ both restrict to $V_{\frac{k+1}{2}}/V_{\frac{k-1}{2}}$ and to $W_{\frac{l+1}{2}}/W_{\frac{l-1}{2}}$ which must then have even rank $2p'$ and $2q'$ respectively ($p'\leq p$ and $q'\leq q$). Hence, if $$\gamma_{\frac{l+1}{2}}=-((\omega^W)^{-1}_{\frac{l+1}{2}}\otimes 1_K)( \beta^t_{\frac{l+1}{2}}\otimes 1_K)\omega^V_{\frac{k+1}{2}}$$ then $$(V_{\frac{k+1}{2}}/V_{\frac{k-1}{2}},\Omega_V,W_{\frac{l+1}{2}}/W_{\frac{l-1}{2}},\Omega_W,\beta_{\frac{l+1}{2}},\gamma_{\frac{l+1}{2}})$$ is an $\Sp(2p',2q')$-Higgs bundle, being an $\Sp(2p'+2q')$-Higgs bundle if 
$\beta_{\frac{l+1}{2}}=0$.

Similarly, if $\lambda_i\neq\mu_j$, then $\beta$ does not restrict to $\beta_j:W_j/W_{j-1}\to V_i/V_{i-1}\otimes K$, and if $i\neq \frac{k+1}{2}$ and $j\neq\frac{l+1}{2}$ then the symplectic forms $\Omega_V$ and $\Omega_W$ also do not restrict to $V_j/V_{j-1}$ and $W_j/W_{j-1}$ respectively. So $(V_j/V_{j-1},0)$ is a $\U(p')$-Higgs bundle ($0$ meaning the zero Higgs field) and $(W_j/W_{j-1},0)$ is a $\U(q')$-Higgs bundle.
If still $\lambda_i\neq\mu_j$ but $i=\frac{k+1}{2}$ (hence $k$ is odd) then the only difference is that $\Omega_V$ restricts to $V_j/V_{j-1}$ so its rank $p'$ must be even. Hence $(V_j/V_{j-1},\Omega_V,0)$ is an $\Sp(p')$-Higgs bundle. Similarly, one obtains an $\Sp(q')$-Higgs bundle if $\lambda_i\neq\mu_j$ but $j=\frac{l+1}{2}$ (hence $l$ is odd).
\endproof

For convenience of the reader, we give an example of a strictly polystable $\Sp(2p,2q)$-Higgs bundle. Let $(V,\Omega_V,W,\Omega_W,\beta,\gamma)$ be a stable $\Sp(2p',2q')$-Higgs bundle. Then clearly $$(V,\Omega_V,W,\Omega_W,\beta,\gamma)\oplus (V,\Omega_V,W,\Omega_W,\beta,\gamma)=(V\oplus V,\Omega_V\oplus\Omega_V,W\oplus W,\Omega_W\oplus\Omega_W,\beta\oplus\beta,\gamma\oplus\gamma)$$ is a polystable $\Sp(4p',4q')$-Higgs bundle. Now, the inclusions $V\hookrightarrow V\oplus V$ and $W\hookrightarrow W\oplus W$ given respectively by $v\mapsto(v,\sqrt{-1}v)$ and $w\mapsto(w,\sqrt{-1}w)$, yield a pair of isotropic subbundles of $(V\oplus V,\Omega_V\oplus\Omega_V)$ and of $(W\oplus W,\Omega_W\oplus\Omega_W)$ such that $(\beta\oplus\beta)(W\oplus W)\subset(V\oplus V)\otimes K$. Moreover, $\deg(V)+\deg(W)=0$. So $(V\oplus V,\Omega_V\oplus\Omega_V,W\oplus W,\Omega_W\oplus\Omega_W,\beta\oplus\beta,\gamma\oplus\gamma)$ is not stable.
This is the reason why we impose non-isomorphic summands in Corollary \ref{cor:direct sum of stable}.

\section{The Hitchin proper functional and the minimal subvarieties}\label{morse quadruples}

Here we use the method introduced by Hitchin in \cite{hitchin:1987} to
study the topology of the moduli spaces of $G$-Higgs bundles, applying it to $\M_{\Sp(2p,2q)}$.

Let $$f:\M_{\Sp(2p,2q)}\longrightarrow\RR$$ be given by
\begin{equation}\label{eq:Hitfunc}
f(V,\Omega_V,W,\Omega_W,\beta,\gamma)=\|\beta\|_{L^2}^2+\|\gamma\|_{L^2}^2=\int_{X}|\beta|^2+|\gamma|^2\mathrm{dvol}.
\end{equation}
This function $f$ is known as the \emph{Hitchin function}.

\begin{remark}
The definition of the Hitchin function $f$ implicitly uses the correspondence between $\M_{\Sp(2p,2q)}$ and the moduli space of solutions of the so-called Hitchin equations. In fact the Hitchin-Kobayashi correspondence states that an $\Sp(2p,2q)$-Higgs bundle $(V,\Omega_V,W,\Omega_W,\beta,\gamma)$ is polystable if and only if it admits a reduction to the maximal compact $\Sp(2p)\times\Sp(2q)$ i.e. a metric $h$ on $(V,\Omega_V,W,\Omega_W)$ such that the Hitchin equations are verified. This is the so-called harmonic metric. For more details, see \cite{corlette:1988,donaldson:1987,hitchin:1987,simpson:1992} and, more recently and in this generality, \cite{bradlow-garcia-prada-mundet:2003,garcia-gothen-mundet:2008}.
Thus, here we are using the harmonic metric on $(V,\Omega_V,W,\Omega_W)$ to define
$\|\beta\|_{L^2}$ and $\|\gamma\|_{L^2}$. 
\end{remark}

Using the Uhlenbeck weak compactness theorem, one can prove \cite{hitchin:1987, hitchin:1992} that the function $f$ is proper and therefore it attains a minimum on each closed subspace $\M'$ of $\M_{\Sp(2p,2q)}$. Moreover, from general topology, one knows that if the subspace of local minima of $f$ on $\M'$ is connected then so is $\M'$.

\begin{proposition}\label{proper}
Let $\M'\subseteq\M_{\Sp(2p,2q)}$ be a closed subspace and let $\N'\subset\M'$ be the subspace of local minima of $f$ on $\M'$. If $\N'$ is connected then so is $\M'$.
\end{proposition}

The idea is to have a detailed description of the subspace of local minima of $f$, enough to draw conclusions about its connectedness. The way this is achieved was first found by Hitchin in \cite{hitchin:1987} and \cite{hitchin:1992}, and then was applied for several cases \cite{bradlow-garcia-prada-gothen:2003,bradlow-garcia-prada-gothen:2005,bradlow-garcia-prada-gothen:2008,garcia-gothen-mundet:2008 II,garcia-oliveira:2012}. This is hence by now a standard method, so we will only sketch it.

The analysis of the local minima of $f$ is done separately for smooth and non-smooth points.
From Proposition \ref{prop:stable and simple imply smooth} we know that a stable and simple $\Sp(2p,2q)$-Higgs bundle represents a smooth point on $\M_{\Sp(2p,2q)}$. So, we will carry the analysis of local minima of $f$ separately for stable and simple, hence smooth, $\Sp(2p,2q)$-Higgs bundles and for those $\Sp(2p,2q)$-Higgs bundle which may not be stable or simple.

\subsection{Stable and simple local minima}

 In this section we only consider stable and simple $\Sp(2p,2q)$-Higgs bundles.
The restriction of the Hitchin function $f$ to points represented by such Higgs bundles is a moment map for the Hamiltonian
$S^1$-action on $\M_{\Sp(2p,2q)}$ given by
$$e^{\sqrt{-1}\theta}\cdot(V,\Omega_V,W,\Omega_W,\beta,\gamma)=(V,\Omega_V,W,\Omega_W,e^{\sqrt{-1}\theta}\beta,e^{\sqrt{-1}\theta}\gamma).$$
A point of $\M_{\Sp(2p,2q)}$ represented by a stable and simple $\Sp(2p,2q)$-Higgs bundle is thus a
critical point of $f$ if and only if it is a fixed point of the
$S^1$-action. Let us then study the fixed point set of the given
action.

Let $(V,\Omega_V,W,\Omega_W,\beta,\gamma)$ represent a stable and simple (hence smooth) fixed point. Then either $\beta=0$ (hence $\gamma=0$)
or, since the action is on $\M_{\Sp(2p,2q)}$, there is a
one-parameter family of gauge transformations
$g(t)$ such that
$$g(t)\cdot(V,\Omega_V,W,\Omega_W,\beta,\gamma)=(V,\Omega_V,W,\Omega_W,e^{\sqrt{-1}t}\beta,e^{\sqrt{-1}t}\gamma).$$

In the latter case, let
$$\psi:=\frac{d}{d t}g(t)|_{t=0}$$
be the infinitesimal gauge
transformation generating this family. Let $V_\lambda$'s and $W_\mu$'s be the eigenbundles of $\psi=(\psi_V,\psi_W)$, where $\psi_V$ and $\psi_W$ are the induced infinitesimal gauge transformations of $V$ and $W$ and $\lambda,\mu\in\RR$. Over $V_\lambda$ and over $W_\mu$,
\begin{equation}\label{psi over V and W}
\psi_V|_{V_\lambda}=\sqrt{-1}\lambda\in\CC,\hspace{1cm}\psi_W|_{W_\mu}=\sqrt{-1}\mu\in\CC.
\end{equation}
Then, from \cite{hitchin:1987,hitchin:1992,simpson:1992}, it follows that $(V,\Omega_V,W,\Omega_W,\beta,\gamma)$ is what is called a \emph{complex variation of
Hodge structure} or a \emph{Hodge bundle}. This means that
\begin{equation}\label{eq:Hodge bundles}
V\cong\bigoplus_\lambda V_\lambda,\hspace{1cm}W\cong\bigoplus_\mu W_\mu.
\end{equation}
Moreover,
\begin{equation}\label{eq:betagamma in Hodge}
\ad(\psi)(\beta,\gamma)=\sqrt{-1}(\beta,\gamma).
\end{equation}
Hence, for each $\lambda,\mu$, $$\beta(W_\mu)\subset V_{\mu+1}\otimes K,\hspace{1cm}\gamma(V_\lambda)\subset W_{\lambda+1}\otimes K$$ so 
\begin{equation}\label{eq:Hodge decomposition of varphi}
\begin{split}
\beta&=\sum\beta_\mu,\hspace{1cm}\beta_\mu:=\beta|_{W_\mu}:W_\mu\longrightarrow V_{\mu+1}\otimes K,\\
\gamma&=\sum\gamma_\lambda,\hspace{1cm}\gamma_\lambda:=\gamma|_{V_\lambda}:V_\lambda\longrightarrow W_{\lambda+1}\otimes K.
\end{split}
\end{equation}
This tells us also that the eigenvalues of $\psi_V$ differ by one, as well as the ones of $\psi_W$, and that there is a tight connection between the eigenvalues of $\psi_V$ and of $\psi_W$.

Since $\psi=(\psi_V,\psi_W)$ locally takes values in $\mathfrak{sp}(2p)\oplus\mathfrak{sp}(2q)$, then using \eqref{psi over V and W} one can prove that
 $V_\lambda$ and $V_{\lambda'}$ are orthogonal under $\Omega_V$ unless $\lambda+\lambda'=0$ and that $W_\mu$ and $W_{\mu'}$ are also orthogonal under $\Omega_W$ unless $\mu+\mu'=0$. Therefore $\omega^V:V\to V^*$ and $\omega^W:W\to W^*$ yield isomorphisms
$$\omega^V_\lambda:=\omega^V|_{V_\lambda}:V_\lambda\stackrel{\cong}{\longrightarrow}V_{-\lambda}^*,\hspace{1cm}\omega^W_\mu:=\omega^W|_{W_\mu}:W_\mu\stackrel{\cong}{\longrightarrow}W_{-\mu}^*$$
satisfying
\begin{equation}\label{eq:skewsymmsymplVW}
(\omega^V_\lambda)^t=-\omega^V_{-\lambda},\hspace{1cm} (\omega^W_\mu)^t=-\omega^W_{-\mu}.
\end{equation}
In particular, this means that the decompositions \eqref{eq:Hodge bundles} are such that
\begin{equation}\label{eq:decomposition of V and  W} 
\begin{split}
V&=\bigoplus_{\lambda=-m}^m V_\lambda=V_{-m}\oplus V_{-m+1}\oplus\dots\oplus V_{m-1}\oplus V_m,\\
W&=\bigoplus_{\mu=-n}^nW_\mu=W_{-n}\oplus W_{-n+1}\oplus\dots\oplus W_{n-1}\oplus W_n
\end{split}
\end{equation}
for some $m,n\geq 1/2$ integers or half-integers. Notice that we also have $(\omega^W_{\lambda+1}\otimes 1_K)\gamma_\lambda=(-\beta_{-\lambda-1}^t\otimes 1_K)\omega^V_\lambda$,  for every $\lambda$.

\begin{remark}\label{rem:type of indices}
Notice that, if $\beta\neq 0$, \eqref{eq:Hodge decomposition of varphi} implies that both $m$ and $n$ must be of the same type, i.e., either both integers (if the number of summands in both sums of \eqref{eq:decomposition of V and  W} is odd) or both half-integers (if the number of summands in both sums of \eqref{eq:decomposition of V and  W} is even).
\end{remark}

Now, the Cartan decomposition of $\liegc$ induces a decomposition of vector bundles
$$E(\liegc)=E(\liehc)\oplus E(\liemc)$$
where $E(\liegc)$ (resp. $E(\liehc)$) is the adjoint bundle associated to the adjoint representation of $H^\CC$ on $\liegc$ (resp. $\liehc$).
For the group $\Sp(2p,2q)$, we have
$$E(\liegc)=\Lambda^2_{\Omega_V\oplus\Omega_W}(V\oplus W)$$ where $\Omega_V\oplus\Omega_W$ is the symplectic form on $V\oplus W$ canonically defined by $\Omega_V$ and $\Omega_W$. 
Also, we already know that
$$E(\liehc)=\Lambda_{\Omega_V}^2V\oplus\Lambda_{\Omega_W}^2W$$ and, from \eqref{eq:condbetagamma}, that
\begin{equation}\label{eq:E(mc)}
E(\liemc)=\{(f,g)\in\Hom(W,V)\oplus\Hom(V,W)\st\omega^Wg=-f^t\omega^V\}\cong\Hom(W,V)
\end{equation}
Let
\begin{equation}\label{eq:Cartanvectorbundle}
\theta:\Lambda^2_{\Omega_V\oplus\Omega_W}(V\oplus W)\longrightarrow\Lambda^2_{\Omega_V\oplus\Omega_W}(V\oplus W)
\end{equation} be the involution in $E(\liegc)=\Lambda^2_{\Omega_V\oplus\Omega_W}(V\oplus W)$ defining the above decomposition. It is, of course, induced by the Cartan involution of the Lie algebra, $\theta$ defined in \eqref{eq:Cartan involution}.
Its $+1$-eigenbundle is $\Lambda_{\Omega_V}^2V\oplus\Lambda_{\Omega_W}^2W$ and its
$-1$-eigenbundle is $E(\liemc)$ as in \eqref{eq:E(mc)}.

Decompositions \eqref{eq:decomposition of V and  W} of $V$ and $W$ yield also a decomposition of $$\End(V\oplus W)=\End(V)\oplus\End(W)\oplus\Hom(W,V)\oplus\Hom(V,W)$$ as follows:
$$\End(V)=\bigoplus_{k=-2m}^{2m}\bigoplus_{\lambda'-\lambda=k}\Hom(V_\lambda,V_{\lambda'}),
\hspace{.5cm}
\End(W)=\bigoplus_{k=-2n}^{2n}\bigoplus_{\mu'-\mu=k}\Hom(W_\mu,W_{\mu'}),
$$
and
$$\Hom(W,V)=\bigoplus_{k=-m-n}^{m+n}\bigoplus_{\lambda-\mu=k}\Hom(W_\mu,V_\lambda),
\hspace{.5cm} \Hom(V,W)=\bigoplus_{k=-m-n}^{m+n}\bigoplus_{\mu-\lambda=k}\Hom(V_\lambda,W_\mu).
$$
From Remark \ref{rem:type of indices}, the index $k$ in all these sums goes over integers and not half-integers.

Write $k_M:=2\max\{m,n\}$. For each $k\in\{-k_M,\ldots,k_M\}$, let $E(\liegc)_k$ be the subbundle of $E(\liegc)$ defined as
$$E(\liegc)_k:=\Big(\bigoplus_{\mu-\lambda=k}\End(V_\lambda\oplus W_\lambda,V_\mu\oplus W_\mu)\Big)\cap\Lambda^2_{\Omega_V\oplus\Omega_W}(V\oplus W).$$
We have therefore, that 
\begin{equation}\label{eq:weight composition liegc}
\Lambda^2_{\Omega_V\oplus\Omega_W}(V\oplus W)=E(\liegc)=\bigoplus_{k=-k_M}^{k_M}E(\liegc)_k.
\end{equation}
From \eqref{psi over V and W}, $E(\liegc)_k$ is the $\sqrt{-1}k$-eigenbundle for the adjoint action
$$\ad(\psi):\Lambda^2_{\Omega_V\oplus\Omega_W}(V\oplus W)\to\Lambda^2_{\Omega_V\oplus\Omega_W}(V\oplus W)$$
of $\psi$. We say that $E(\liegc)_k$ is the subspace of
$\Lambda^2_{\Omega_V\oplus\Omega_W}(V\oplus W)$ with \emph{weight} $k$.

Now, given an element $v$ in $E(\liegc)_k$, we have $\ad(\psi)(v)=\sqrt{-1}kv$. On the other hand, if $\theta$ is given by \eqref{eq:Cartanvectorbundle}, $$\ad(\psi)(\theta v)=[\psi,\theta v]=\theta[\theta\psi,v]=\theta[\psi,v]=\sqrt{-1}k\theta v$$ because $\theta$ in \eqref{eq:Cartanvectorbundle} is induced by the Lie algebra Cartan involution \eqref{eq:Cartan involution} and because $\psi$ takes values in $\liehc$.
So, we conclude that $\theta$ in \eqref{eq:Cartanvectorbundle} restricts to an involution, $\theta_k:E(\liegc)_k\longrightarrow E(\liegc)_k$, on $E(\liegc)_k$.
Hence we can write $$E(\liegc)_k=E(\liehc)_k\oplus E(\liemc)_k$$ where $E(\liehc)_k,E(\liemc)_k$ are the $\pm 1$ eigenbundles of $\theta_k$. $E(\liehc)_k$ is given by
\begin{equation*}
\begin{split}
E(\liehc)_k&:=E(\liegc)_k\cap E(\liehc)\\
&=\bigoplus_{\mu-\lambda=k}\bigg(\big(\Hom(V_\lambda,V_\mu)\cap\Lambda_{\Omega_V}^2V\big)\oplus\big(\Hom(W_\lambda,W_\mu)\cap\Lambda_{\Omega_W}^2W\big)\bigg)
\end{split}
\end{equation*}
thus, 
\begin{equation}\label{eq:weight composition liehc}
\Lambda_{\Omega_V}^2V\oplus\Lambda_{\Omega_W}^2W=E(\liehc)=\bigoplus_{k=-k_M}^{k_M}
E(\liehc)_k.
\end{equation} 
On the other hand,
\begin{equation*}
\begin{split}
E(\liemc)_k&:=E(\liegc)_k\cap E(\liemc)\\
&=\Big(\bigoplus_{\mu-\lambda=k}\Hom(W_\lambda,V_\mu)\oplus\Hom(V_\lambda,W_\mu)\Big)\cap\\
&\quad\cap\{(f,g)\in\Hom(W,V)\oplus\Hom(V,W)\st \omega^Wg=-f^t\omega^V\},
\end{split}
\end{equation*} thus 
\begin{equation}\label{eq:weight composition liemc}
E(\liemc)=\bigoplus_{k=-k_M}^{k_M} E(\liemc)_k.
\end{equation}
Observe that, by \eqref{eq:betagamma in Hodge}, $(\beta,\gamma)\in H^0(E(\liemc)_1\otimes K)$.

Notice that if $(f,g)\in E(\liegc)_k$ is such that $$f:=\sum_{\mu-\lambda=k} f_{\lambda,\mu}\hspace{.3cm}\text{and}\hspace{.3cm}g:=\sum_{\mu-\lambda=k} g_{\lambda,\mu},$$ with $f_{\lambda,\mu}\in\Hom(V_\lambda,V_\mu)$ and $g_{\lambda,\mu}\in\Hom(W_\lambda,W_\mu)$, then it follows from the definition of $E(\liehc)_k$ and from \eqref{eq:skewsymmsymplVW} that $(f,g)\in E(\liehc)_k$ if and only if
\begin{equation}\label{eq:condiE(h)}
\omega^V_{-\lambda}f_{-\mu,-\lambda}=-f_{\lambda,\mu}^t\omega^V_{-\mu},
\hspace{1cm}
\omega^W_{-\lambda}g_{-\mu,-\lambda}=-g_{\lambda,\mu}^t\omega^W_{-\mu},
\end{equation}
for every $\mu,\lambda$ with $\mu-\lambda=k$.

The map $\ad(\beta,\gamma)$ interchanges $E(\liehc)$ with $E(\liemc)$ and
therefore maps $E(\liehc)_k$ to $E(\liemc)_{k+1}\otimes K$ and $E(\liemc)_k$ to $E(\liehc)_{k+1}\otimes K$. So, for each
$k$, we have a weight $k$ subcomplex of the complex $C^\bullet=C^\bullet(V,\Omega_V,W,\Omega_W,\beta,\gamma)$ defined in \eqref{eq:defcomplex}:
\begin{equation}\label{eq:subcomplex}
C^\bullet_k=C^\bullet_k(V,\Omega_V,W,\Omega_W,\beta,\gamma):E(\liehc)_k\xrightarrow{\ad(\beta,\gamma)_k}E(\liemc)_{k+1}\otimes K.
\end{equation}

The following result is fundamental for the description of the smooth local minima of $f$. This is basically  \cite[Lemma 3.11]{bradlow-garcia-prada-gothen:2008} (see also \cite[Proposition 4.4]{bradlow-garcia-prada-gothen:2003}). Although the proof in those papers is for $\GL(n,\CC)$ and $\U(p,q)$-Higgs bundles, the same argument works in the general setting of $G$-Higgs bundles (see \cite[Remark 4.16]{bradlow-garcia-prada-gothen:2003}): the key facts are that for a stable $G$-Higgs bundle, $(E_{H^\CC},\varphi)$, the Higgs vector bundle $(E_{H^\CC}\times_{\Ad}\liegc,\ad(\varphi))$ is semistable, and that there is a natural $\ad$-invariant isomorphism $E_{H^\CC}\times_{\Ad}\liegc\cong(E_{H^\CC}\times_{\Ad}\liegc)^*$ given by an invariant pairing on $\liegc$, such as the Killing form. 

\begin{theorem}\label{ad}
Let $(V,\Omega_V,W,\Omega_W,\beta,\gamma)\in\M_{\Sp(2p,2q)}$ be a stable and simple critical point of the Hitchin function $f$. Then
$(V,\Omega_V,W,\Omega_W,\beta,\gamma)$ is a local minimum if and only if either $\beta=\gamma=0$ or $\ad(\beta,\gamma)_k$ in \eqref{eq:subcomplex} is an isomorphism for all $k\geq 1$.
\end{theorem}

Using this, one can now describe the smooth local minima of the Hitchin function $f$.

\begin{proposition}\label{min}
Let the $\Sp(2p,2q)$-Higgs bundle $(V,\Omega_V,W,\Omega_W,\beta,\gamma)$ be a critical point of the Hitchin function
$f$ such that $(V,\Omega_V,W,\Omega_W,\beta,\gamma)$ is stable and simple (hence smooth). Then $(V,\Omega_V,W,\Omega_W,\beta,\gamma)$ represents a local minimum if and only if $\beta=\gamma=0$.
\end{proposition}
\proof From the definition of $f$ in \eqref{eq:Hitfunc}, if $\beta=\gamma=0$, then clearly $(V,\Omega_V,W,\Omega_W,\beta,\gamma)$ represents a local minimum.

Let us prove the converse. Recall that $\beta$ and $\gamma$ are related by
\eqref{eq:condbetagamma}. In particular, $\beta=0$ if and only if
$\gamma=0$. Suppose that $(V,\Omega_V,W,\Omega_W,\beta,\gamma)$ is a critical
point of $f$ with $\beta\neq 0$. Hence, as explained above, we have the
decompositions 
\eqref{eq:decomposition of V and  W} 
of $V$ and $W$ and also the weight decompositions \eqref{eq:weight composition liegc}, \eqref{eq:weight composition liehc} and \eqref{eq:weight composition liemc} of $E(\liegc)$, $E(\liehc)$ and of $E(\liemc)$ respectively.

Recall that $k_M=2\max\{m,n\}$ is the highest weight. Suppose that $m\geq n$, so that $k_M=2m\geq 1$. 
Take the complex $C^\bullet_{2m}$ for this highest weight. Then from Theorem \ref{ad}, we must have an isomorphism
$$\ad(\beta,\gamma)_{2m}:E(\liehc)_{2m}\longrightarrow E(\liemc)_{2m+1}\otimes K.$$
Since we are taking the highest weight, $E(\liemc)_{2m+1}\otimes K=0$. However $E(\liehc)_{2m}\neq 0$.
Indeed, $$E(\liehc)_{2m}=\Hom(V_{-m},V_m)\cap\Lambda^2_{\Omega_V}V.$$ Since $V_{-m}\neq 0$, then it is possible to find a non-zero map $g:V_{-m}\to V_{-m}^*$ which is symmetric: $g^t=g$. Now, consider $$(\omega_m^V)^{-1}g:V_{-m}\longrightarrow V_m.$$ It follows from the symmetry of $g$, from \eqref{eq:skewsymmsymplVW} and from \eqref{eq:condiE(h)} that $(\omega_m^V)^{-1}g$ is indeed a non-zero element of  $E(\liehc)_{2m}$.

So, $\ad(\beta,\gamma)_{2m}$ is not an isomorphism and by the previous theorem, $(V,\Omega_V,W,\Omega_W,\beta,\gamma)$ is not a local minimum of $f$.
\endproof

\subsection{Local minima in all moduli space}

In \cite{hitchin:1992}, it was observed that the Hitchin function is additive with respect to direct sum of Higgs bundles. In our case this means that $$f(\bigoplus(V_i,\Omega_{V_i},W_i,\Omega_{W_i},\beta_i,\gamma_i))=\sum f(V_i,\Omega_{V_i},W_i,\Omega_{W_i},\beta_i,\gamma_i).$$
Hence, the following proposition is immediate from the previous proposition and from Proposition \ref{prop:description of stable and non-simple}.

\begin{proposition}\label{min for stable and not simple}
A stable $\Sp(2p,2q)$-Higgs bundle $(V,\Omega_V,W,\Omega_W,\beta,\gamma)$ represents a local minimum of $f$ if and only if $\beta=\gamma=0$.
\end{proposition}

In order to have a description of the subvariety of local minima of $f$, it remains to deal with the strictly polystable $\Sp(2p,2q)$-Higgs bundles.

\begin{theorem}\label{min for polystable}
A polystable $\Sp(2p,2q)$-Higgs bundle $(V,\Omega_V,W,\Omega_W,\beta,\gamma)$ represents a local minimum if and
only if $\beta=\gamma=0$.
\end{theorem}
\proof
From Theorem \ref{thm:description of polystable} we know that a polystable minima of $f$ decomposes as a direct sum of stable $G_i$-Higgs bundles where $G_\alpha=\Sp(2p_\alpha,2q_\alpha),\,\Sp(2n_\alpha),\,\U(p_\alpha,q_\alpha)$ or $\U(n_\alpha)$, with $p_\alpha\leq p,\, q_\alpha\leq q,\, n_\alpha\leq p+q$. For the compact groups $\Sp(2n_\alpha)$ or $\U(n_\alpha)$ it is clear that the local minima of $f$ on the corresponding lower rank moduli spaces must have zero Higgs field. For stable $\Sp(2p_\alpha,2q_\alpha)$-Higgs bundles, we apply Proposition \ref{min for stable and not simple} to obtain the same conclusion.
The case of $\U(p_\alpha,q_\alpha)$ is, however, more complicated, due to the fact that
the stable $\U(p_\alpha,q_\alpha)$-Higgs bundles $(V',W',\beta',\gamma')$ which are
local minima of the Hitchin function in the moduli space $\M_{\U(p_\alpha,q_\alpha)}$
have only $\beta'=0$ or $\gamma'=0$ --- both of them are zero only when the
degrees of $V'$ and $W'$ are zero ---
(see \cite[Theorem 4.6]{bradlow-garcia-prada-gothen:2003}). Suppose, without loss of generality, that $\gamma'=0$. Then, from \cite[Proposition 4.8]{bradlow-garcia-prada-gothen:2003}, we must have $\deg(V')\leq 0$ (also, $\deg(W')=-\deg(V')>0$ --- this follows from the description of the $\U(p_\alpha,q_\alpha)$ which may occur in Theorem \ref{thm:description of polystable}).  Then $(V',W',\beta',0)$ is a stable local minima of $f$ in  $\M_{\U(p_\alpha,q_\alpha)}$, so the corresponding strictly polystable $\Sp(2p_\alpha,2q_\alpha)$-Higgs bundle 
\begin{equation}\label{eq:Sp from U}
(V'\oplus V'^*,\langle\, , \rangle_{V'},W'\oplus W'^*,\langle\, , \rangle_{W'},\beta',\beta'^t)
\end{equation} is a potential local minima of $f$ in $\M_{\Sp(2p_\alpha,2q_\alpha)}$ without $\beta'=0$.
However, in Lemma \ref{lemma:deformation} below, we show that such $\Sp(2p_\alpha,2q_\alpha)$-Higgs bundle can always be continuously deformed into a stable $\Sp(2p_\alpha,2q_\alpha)$-Higgs bundle, which together with Proposition \ref{min for stable and not simple} proves that \eqref{eq:Sp from U} is not a local minimum of $f$. This completes the proof.
\endproof

\begin{lemma}\label{lemma:deformation}
Let $(V',W',\beta',0)$ be a stable $\U(p',q')$-Higgs bundle. Then the corresponding strictly polystable $\Sp(2p',2q')$-Higgs bundle $(V'\oplus V'^*,\langle\, ,\rangle_V,W'\oplus W'^*,\langle\, ,\rangle_W,\beta',-\beta'^t)$ can be deformed to a stable $\Sp(2p',2q')$-Higgs bundle. 
\end{lemma}
\proof
Since the part in $H^0(\Hom(V',W')\otimes K))$ of the $\U(p',q')$-Higgs bundle $(V',W',\beta',0)$ is zero then, by \cite[Proposition 4.8]{bradlow-garcia-prada-gothen:2003}, the degree of $V'$ must be negative. Write $\deg(V')=d'<0$, so that $\deg(W')=-d'>0$.
Write also, $\beta=\beta'$ and $\gamma=-\beta'^t$, so that the corresponding $\Sp(2p',2q')$-Higgs bundle is $$(V'\oplus V'^*,\langle\, ,\rangle_V,W'\oplus W'^*,\langle\, ,\rangle_W,\beta,\gamma)$$ which is of course strictly polystable. In order to deform it, we will first make use of non-trivial extensions of $V'^*$ by $V'$ and of $W'$ by $W'^*$.
These are of course parametrized by non-zero elements of $H^1(V'\otimes V')$ and $H^1(W'^*\otimes W'^*)$. However, since we want the vector bundles given by these extensions to carry symplectic forms, we must (cf. \cite[Criterion 2.1]{hitching:2007}) consider only extensions parametrized by non-zero elements of $H^1(S^2V')$ and $H^1(S^2W'^*)$. Since $d'\leq 0$ and $g\geq 2$, it follows from Riemann-Roch that $\dim H^1(S^2V')>0$ and $\dim H^1(S^2W'^*)>0$. So we can choose any non-zero elements 
\begin{equation}\label{eq:extension classes}
\eta_V\in H^1(S^2V'),\ \eta_W\in H^1(S^2W'^*),
\end{equation} defining then symplectic vector bundles $(F,\Omega_F)$ and $(G,\Omega_G)$ such that
$$0\longrightarrow V'\stackrel{\ i_{V'}\ }{\longrightarrow} F\stackrel{\ p_{V'}\ }{\longrightarrow} V'^*\longrightarrow 0\hspace{.5cm}\text{ and }\hspace{.5cm}0\longrightarrow W'^*\stackrel{\ i_{W'}\ }{\longrightarrow}  G\stackrel{\ p_{W'}\ }{\longrightarrow} W'\longrightarrow 0.$$
Notice that in this way, $V'$ and $W'^*$ are Lagrangian subbundles of $F$ and $G$ respectively. Moreover, if $\omega^F:F\to F^*$ and $\omega^G:G\to G^*$ are the isomorphisms coming from $\Omega_F$ and $\Omega_G$, then
\begin{equation}\label{eq:extensions and symplectic forms}
p_{W'}^t=\omega^Gi_{W'},\hspace{.5cm}p_{V'}=i_{V'}^t\omega^F.
\end{equation}
Define now the $\Sp(2p',2q')$-Higgs bundle $$(F,\Omega_F,G,\Omega_G,\tilde\beta,\tilde\gamma)$$ where $$\tilde\beta=i_{V'}\beta p_{W'}, \ \tilde\gamma=i_{W'}\gamma p_{V'}.$$ From \eqref{eq:extensions and symplectic forms}, this is equivalent to $$\tilde\gamma=-(\omega^G)^{-1}\tilde\beta^t\omega^F$$ as wanted.

We have to prove that $(F,\Omega_F,G,\Omega_G,\tilde\beta,\tilde\gamma)$ is stable. So, let $A\subset F$ and $B\subset G$ be any isotropic subbundles such that 
$$\tilde\beta(B)\subset A\otimes K,$$ 
which is equivalent to
$$\tilde\gamma(A)\subset B\otimes K.$$
Since semistability is an open condition, then $(F,\Omega_F,G,\Omega_G,\tilde\beta,\tilde\gamma)$ is semistable, so $\deg(A)+\deg(B)\leq 0$. We want to prove that 
\begin{equation}\label{eq:degA+degB<0}
\deg(A)+\deg(B)<0.
\end{equation}

Let $A'$ be the kernel of $p_{V'}$ restricted to $A$. It is a vector subbundle of $A$ (recall that we are on a Riemann surface). Let $A''=A/A'$ be the quotient bundle. Define similarly $B'\subset B$ and $B''=B/B'$.
So we have the following commutative diagrams (where to simplify we are not including $K$)
$$
\xymatrix{
&0\ar[r]&A'\ar@{^{(}->}[d]\ar[r]^{i_{V'}|_{A'}}&A\ar@{^{(}->}[d]\ar[r]^{p_{V'}|_A}&A''\ar@{^{(}->}[d]\ar[r]&0\\
&0\ar[r]&V'\ar[r]^{i_{V'}}&F\ar[r]^{p_{V'}}\ar@/^-.5pc/_{\tilde\gamma}[d]&V'^*\ar[r]&0\\
&0\ar[r]&W'^*\ar[r]^{i_{W'}}&G\ar[r]^{p_{W'}}\ar@/^-.5pc/_{\tilde\beta}[u]&W'\ar[r]&0\\
&0\ar[r]&B'\ar@{^{(}->}[u]\ar[r]^{i_{W'}|_{B'}}&B\ar@{^{(}->}[u]\ar[r]^{p_{W'}|_{B}}&B''\ar@{^{(}->}[u]\ar[r]&0.}
$$
It is easy to see that 
\begin{equation}\label{eq:invariant subbundles}
\begin{split}
\tilde\beta(B)\subset A\otimes K &\Longleftrightarrow\beta(B'')\subset A'\otimes K,\\
\tilde\gamma(A)\subset B\otimes K &\Longleftrightarrow\gamma(A'')\subset B'\otimes K
\end{split}
\end{equation}
where we recall that, by definition, $\beta=\beta'\in H^0 (\Hom(W',V')\otimes K)$ is the map in the $\U(p',q')$-Higgs bundle $(V',W',\beta',0)$ we started with, and also by definition, $\gamma=-\beta'^t\in H^0 (\Hom(V'^*,W'^*)\otimes K)$ is the map on the dual $\U(p',q')$-Higgs bundle $(V'^*,W'^*,0,-\beta'^t)$.
Notice that since $(V',W',\beta',0)$ is stable then $(V'^*,W'^*,0,-\beta'^t)$ is stable as well.
Now, $$\deg(A)+\deg(B)=0\Longleftrightarrow\deg(A')+\deg(B'')=-(\deg(A'')+\deg(B')).$$
This together with \eqref{eq:invariant subbundles} contradicts the stability
of $(V',W',\beta',0)$ and $(V'^*,W'^*,0,-\beta'^t)$, unless $A'=V'$ and
$B''=W'$, or $A'=B''=0$. However, in this cases we will also reach a contradiction. Indeed, if $A'=V'$ and $B''=W'$, then we must have 
$$
\xymatrix{
&0\ar[r]&W'^*\ar[r]&G\ar[r]&W'\ar[r]&0\\
&0\ar[r]&B'\ar@{^{(}->}[u]\ar[r]&B\ar@{^{(}->}[u]\ar[r]&W'\ar@{=}[u]\ar[r]&0}
$$
and since $B$ is isotropic, its rank is at most $q'$, this being the rank of $W'$. It follows that $B\cong W'$. But this yields a splitting of the non-trivial extension defining $G$, a contradiction.
If instead we have $A'=B''=0$, then we must have $A''=V'^*$ and $B'=W'^*$, and along the same lines this will contradict the non-triviality of the extension defining $F$.
Therefore we must have $\deg(A)+\deg(B)< 0$, proving \eqref{eq:degA+degB<0} and showing that $(F,\Omega_F,G,\Omega_G,\tilde\beta,\tilde\gamma)$ is stable.

Finally, the required deformation is given just by taking $(F_t,\Omega_{F_t},G_t,\Omega_{G_t},\tilde\beta_t,\tilde\gamma_t)$, where $t$ is a parameter in the unit disc, and $(F_t,\Omega_{F_t})$ and $(G_t,\Omega_{G_t})$ are the symplectic vector bundles defined by classes $t\eta_V$ and $t\eta_W$, where $\eta_V,\eta_W$ are given in \eqref{eq:extension classes}, and $\tilde\beta_t,\tilde\gamma_t$ are defined in a similar manner as $\tilde\beta$ and $\tilde\gamma$ above.
\endproof

\section{Connected components of the space of $\Sp(2p,2q)$-Higgs bundles}\label{compM}

From Theorem \ref{min for polystable} we conclude that the subvariety $\N_{\Sp(2p,2q)}$ of local minima of the Hitchin function $f:\M_{\Sp(2p,2q)}\to\RR$ is isomorphic to the moduli space of $\Sp(2p+2q,\CC)$-principal bundles or, 
in the language of Higgs bundles, to the moduli space of $\Sp(2p+2q)$-Higgs bundles.

In \cite{ramanathan:1996a,ramanathan:1996b}, A. Ramanathan has shown that 
if $G$ is a connected reductive group then there is a bijective correspondence 
between $\pi_0$ of the moduli space of $G$-principal bundles  and $\pi_1G$. Hence, 
since $\Sp(2p+2q)$ is simply-connected, it follows therefore that 
the same is true for $\N_{\Sp(2p,2q)}$. So, using Proposition \ref{proper}, we can state our result.

\begin{theorem}\label{thm:main}
 Let $X$ be a compact Riemann surface of genus $g\geq 2$ and let $\M_{\Sp(2p,2q)}$ be the moduli space of $\Sp(2p,2q)$-Higgs bundles. Then $\M_{\Sp(2p,2q)}$ is connected.
\end{theorem}

Regarding the special case of $\Sp(2p,2p)$, our Theorem \ref{thm:main} proves in particular Conjecture 9.1 of \cite{schaposnik:2012}.


\begin{thebibliography}{99}

\bibitem{biswas-ramanan:1994}
{I. Biswas, S. Ramanan}, An infinitesimal study of the
moduli of Hitchin pairs, \textsl{J. London Math.} Soc. (2) \textbf{49}
(1994), 219--231.

\bibitem{bradlow-garcia-prada-gothen:2003}
{S. B. Bradlow, O. García-Prada, P. B. Gothen}, Surface
group representations and $\U(p,q)$-Higgs bundles, \textsl{J. Diff. Geom.}
\textbf{64} (2003), 111--170.

\bibitem{bradlow-garcia-prada-gothen:2005}
{S. B. Bradlow, O. García-Prada, P. B. Gothen}, Maximal
surface group representations in isometry groups of classical
Hermitian symmetric spaces, \textsl{Geometriae Dedicata} \textbf{122} (2006), 185--213.

\bibitem{bradlow-garcia-prada-gothen:2008}
{S. B. Bradlow, O. García-Prada, P. B. Gothen}, Homotopy groups of moduli spaces of representations, \textsl{Topology} \textbf{47} (2008), 203--224.

\bibitem{bradlow-garcia-prada-mundet:2003}
{S. B. Bradlow, O. García-Prada, I. Mundet i Riera},
Relative Hitchin-Kobayashi correspondences for principal
pairs, \textsl{Quart. J. Math.} \textbf{54} (2003), 171--208.

\bibitem{corlette:1988}
{K. Corlette}, Flat $G$-bundles with canonical metrics,
\textsl{J. Diff. Geom.} \textbf{28} (1988), 361--382.

\bibitem{donaldson:1987}
{S. K. Donaldson}, Twisted harmonic maps and self-duality
equations, \textsl{Proc. London Math. Soc.} (3) \textbf{55} (1987),
127--131.

\bibitem{garcia-gothen-mundet:2008}
{O. García-Prada, P. B. Gothen, I. Mundet i Riera}, The Hitchin-Kobayashi correspondence, Higgs pairs and surface group representations, Preprint arXiv:0909.4487v3.

\bibitem{garcia-gothen-mundet:2008 II}
{O. García-Prada, P. B. Gothen, I. Mundet i Riera}, Higgs bundles and surface
group representations in the real symplectic group, 
\textsl{J. Topology}, to appear.

\bibitem{garcia-oliveira:2012}
{O. García-Prada, A. G. Oliveira}, Higgs bundles for the non-compact dual of the unitary group, \textsl{Illinois J. Math.}, to appear.

\bibitem{hitchin:1987}
{N. J. Hitchin}, The self-duality equations on a Riemann
surface, \textsl{Proc. London Math. Soc.} (3) \textbf{55} (1987), 59--126.

\bibitem{hitchin:1992}
{N. J. Hitchin}, Lie groups and Teichm\"{u}ller space,
\textsl{Topology} \textbf{31} (1992), 449--473.

\bibitem{hitching:2005}
{G. H. Hitching}, \emph{Moduli of Symplectic Bundles over Curves}, Doctoral Dissertation, Department of Mathematical Sciences, University of Durham, 2005.

\bibitem{hitching:2007}
{G. H. Hitching}, Subbundles of symplectic and orthogonal vector bundles over curves, \textsl{Math. Nachr.} \textbf{280}, No. 13-14, (2007), 1510--1517.

\bibitem{ramanathan:1975}
{A. Ramanathan}, Stable principal bundles on a compact
Riemann surface, \textsl{Math. Ann.} \textbf{213} (1975), 129--152.

\bibitem{ramanathan:1996a}
{A. Ramanathan}, Moduli for principal bundles over algebraic
curves: I, \textsl{Proc. Indian Acad. Sci. (Math. Sci.)} \textbf{106}
(1996), 301--328.

\bibitem{ramanathan:1996b}
{A. Ramanathan}, Moduli for principal bundles over algebraic
curves: II, \textsl{Proc. Indian Acad. Sci. (Math. Sci.)} \textbf{106}
(1996), 421--449.

\bibitem{schaposnik:2012}
{L. P. Schaposnik}, \textsl{Spectral data for G-Higgs bundles},
D. Phil. Thesis, University of Oxford, 2012. Preprint arXiv:1301.1981.

\bibitem{schmitt:2008}
{A. H. W. Schmitt}, \textsl{Geometric Invariant Theory and Decorated Principal Bundles},
Zurich Lectures in Advanced Mathematics, European Mathematical Society, 2008.

\bibitem{simpson:1988}
{C. T. Simpson}, Constructing variations of Hodge structure using
  Yang-Mills theory and applications to uniformization,
\textsl{J. Amer. Math. Soc.} \textbf{1} (1988), 867--918.

\bibitem{simpson:1992}
{C. T. Simpson}, Higgs bundles and local systems, \textsl{Inst.
Hautes \'{E}tudes Sci. Publ. Math.} \textbf{75} (1992), 5--95.

\end{thebibliography}
\end{document}